\documentclass{amsart}

\usepackage{graphicx}
\usepackage{amssymb, amsmath, amsthm}
\usepackage{geometry}
\usepackage{lipsum}
\usepackage{enumerate}
\usepackage{dsfont}
\usepackage{mathrsfs}
\usepackage{xcolor}
\usepackage{stmaryrd}
\input xy
\xyoption{all}

\usepackage{tikz-cd}

\usepackage{hyperref}

\newtheorem{theorem}{Theorem}[section]
\newtheorem{lemma}[theorem]{Lemma}
\newtheorem{corollary}[theorem]{Corollary}
\newtheorem{proposition}[theorem]{Proposition}
\newtheorem{lemma-definition}[theorem]{Lemma-Definition}

\newtheorem{thmx}{Theorem}

\newtheorem{corx}[thmx]{Corollary}

\theoremstyle{definition}
\newtheorem{definition}[theorem]{Definition}
\newtheorem{example}[theorem]{Example}

\theoremstyle{remark}
\newtheorem{remark}[theorem]{Remark}
\newtheorem{parrafo}[theorem]{\unskip}
\numberwithin{equation}{theorem}

\newcommand{\mc}{\mathcal}

\newcommand{\mb}{\mathbb}
\newcommand{\mf}{\mathfrak}
\newcommand{\Spec}{\operatorname{Spec}}

\newcommand{\on}{\operatorname}
\newcommand{\x}{\mathbf{x}}

\newcommand{\z}{\mathbf{z}}
\newcommand{\e}{\mathbf{e}}

\newcommand{\m}{\mathfrak{m}}

\newcommand{\p}{\mathfrak{p}}

\renewcommand{\d}{\partial}
\newcommand{\Der}{\operatorname{Der}}
\newcommand{\Hom}{\operatorname{Hom}}
\newcommand{\End}{\operatorname{End}}
\newcommand{\Diff}{\operatorname{Diff}}

 \title{Purely Inseparable extensions of rings}

\author{Celia del Buey de Andr\'es}
\address{Dpto. de Matem\'aticas, Universidad Aut\'onoma de Madrid, Ciudad Universitaria de Cantoblanco, 28049 Madrid, Spain}
\email{ceclia.delbuey@uam.es}
\author{Diego Sulca}
\address{CIEM-FAMAF, Universidad Nacional de C\'ordoba, 5000 C\'ordoba, Argentina}
\email{diego.a.sulca@unc.edu.ar}
\begin{document}


\begin{abstract}
    We revisit the concept of special algebras, also known as \textit{purely inseparable ring extensions}. This concept extends
    the notion of purely inseparable field extensions to the more general context of extensions of commutative rings.
    We use differential operators methods to provide a characterization for a ring extension to be purely inseparable in terms of a condition on certain modules of differential operators associated to the ring extension. This approach is also used to re-obtain an already known characterization involving the modules of principal parts.

    Next, given a purely inseparable ring extension $A\subset C$, we prove a Galois correspondence between intermediate rings $A\subset B\subset C$ such that $C$ is a projective $B$-module and $A$-subalgebras  $H\leq \End_A(C)$ that are also (left) $C$-module direct summand, obtaining a generalization of the Jacobson-Bourbaki theorem.

    Finally, given a tower of ring extensions $A\subset B\subset C$, we address the question of whether the fact that two of the three extensions $A\subset C$, $A\subset B$, and $B\subset C$ are purely inseparable implies that the third one is also purely inseparable.
\end{abstract}

\maketitle

\section{Introduction}

\noindent Unless otherwise is stated, all rings in this paper are assumed to be commutative, with identity $1\neq 0$, and of characteristic $p>0$, for a fixed prime $p$.

\medskip

Let $A\subset C$ be a ring extension.
We say that an element $x\in C$ has \textit{finite exponent} over $A$ if there exists an integer $e\geq 0$ such that $x^{p^e}\in A$. 
We say that the extension $A\subset C$ has \textit{finite exponent} if there is an integer $e$ such that every $x\in C$ has exponent at most $e$. We call the minimal $e$ with this property the \textit{exponent} of the extension and we denote it by $\exp(C/A)$. 

\medskip

A field extension $K\subset L$ such that every element $x\in L$ has finite exponent over $K$ is called a \textit{purely inseparable field extension}. 
If $K\subset L$ is a finite purely inseparable field extension, then it has finite exponent, and if $\overline{K}$ is an algebraic closure of $K$, then $L\otimes_K\overline{K}\cong \overline{K}[X_1,\ldots,X_n]/\langle X_1^{p^{e_1}},\ldots,X_n^{p^{e_n}}\rangle$ for some $e_i>0$ (see \cite{Ras71}).
In fact, this is a property that characterizes finite purely inseparable field extensions among finite field extensions.
We will use this characteristic property to extend the notion of purely inseparable extension to the setting of rings.

\begin{definition}\label{def: purely inseparable}
A ring extension $A\subset C$ is called {\em purely inseparable} if
\begin{enumerate}[(i)]
    \item $A\subset C$ has finite exponent, 
    \item $C$ is a finite projective $A$-module, and
    \item for each prime ideal $\p\subset A$, there is a faithfully flat extension $A_\p\to A'$ such that $C_\p\otimes_{A_\p} A'$ is isomorphic to an $A'$-algebra of the form $A'[X_1,\ldots,X_n]/\langle X_1^{p^{e_1}},\ldots,X_n^{p^{e_n}} \rangle$ for some $e_i>0$.
\end{enumerate}
\end{definition}

Purely inseparable ring extensions of exponent one are also called \textit{Galois extensions of exponent one}.
They were introduced by Yuan in \cite{Yuan70/1} in an attempt to extend the Jacobson's Galois theory of purely inseparable field extensions of exponent one (\cite{Jacobson44}) to the setting of rings.
Purely inseparable ring extensions of arbitrary exponent were are also called {\em special algebras} by Pauer in \cite{Pauer78}.  The case of special algebras over a field was previously treated by Rasala in \cite{Ras71}.  The term {\em purely inseparable} to refer to special algebras was used in \cite{FPSan00}. We should mention that this term had been used before with different nuances. Sweedler introduced it in \cite{Sweedler75} for the purpose of extending the definition of purely inseparable extension of fields to ring extensions. Then his definition for inseparability was revised by Holleman in relation with radicial extensions (\cite{Holleman77}) and reformulated into a stronger notion by Sato (\cite{Sato81}).  

\medskip

Regarding the notion of purely inseparable ring extension presented in this paper, there are some characterizations  related to certain modules of Kähler differentials (\cite{Pauer78}), to the modules of principal parts (\cite{PSan99}), and to the modules of derivations in the case of exponent one (\cite{Yuan70/1}). One of the aims of these notes is to revisit those characterizations of purely inseparable ring extensions  
and provide a new equivalence related to the modules of higher order differential operators.

More precisely, let $A\subset C$ be a finite ring extension of  finite exponent.
Then the following conditions are equivalent.
\begin{enumerate}
    \item $A\subset C$ is purely inseparable. 
    \item  The $A[C^{p^e}]$-module of K\"ahler differentials $\Omega_{A[C^{p^e}]/A}$ is projective for all $e\geq 0$ (Pauer, \cite{Pauer78}).
    \item The $C$-module of principal parts $\on{P}_{C/A}^k$ is projective for all $k\geq 0$ (P.\ S.\ de Salas, \cite{PSan99}).
\end{enumerate}

In \cite{PSan99}, an $A$-algebra $C$ is called \textit{differentially homogeneous} if it is flat, of finite presentation, and $\on{P}_{C/A}^k$ is a projective $C$-module for all $k\geq 0$. 
So the equivalence (1)$\Leftrightarrow$(3)
tells us that the purely inseparable $A$-algebras are the differentially homogeneous $A$-algebras of finite exponent.

\medskip

Our first result extends the above list of characterizations of purely inseparable extensions.

\begin{thmx}\label{purely inseparable new characterization}
Let $A\subset C$ be a finite ring extension of finite exponent. Then the following conditions are equivalent:   
\begin{enumerate}
    \item $A\subset C$ is purely inseparable.
    \item[(3')] The $C$-module of principal parts $\on{P}_{C/A}^{p^e}$ is projective for all $0\leq e<\exp(C/A)$.
    \item[(4)]  $C$ is a projective $A[C^{p^e}]$-module  for all $0\leq e< \exp(C/A)$, and the module of differential operators $\on{Diff}_{A}^k(C)$ is a $C$-module direct summand of $\End_A(C)$ for all $k\geq 0$.
    \item[(4')] $C$ is a projective $A[C^{p^e}]$-module and the module of differential operators $\on{Diff}_{A}^{p^{e}}(C)$ is a $C$-module direct summand of $\End_A(C)$ for all $0\leq e< \exp(C/A)$.
\end{enumerate}
\end{thmx}

In order to show these new equivalences, we prove a result on extensions of differential operators. Namely, given a finite extension of finite exponent $A\subset C$ such that $A[C^{p}]\subset C$ is Galois (equivalently, such that $\Omega_{C/A}$ is a projective $C$-module), we prove that every $A$-differential operator on $A[C^{p}]$ of order $\leq k$ can be extended, at least locally, to an $A$-differential operator of $C$ of order $\leq pk$ (see Proposition \ref{prop: extension of differential operators}). An equivalent formulation of this result was proved under much stronger assumptions in \cite[Lemma 3.3]{PSan99} with completely different techniques.

\medskip

We then move to the study of certain subextensions of a purely inseparable ring extension with a view toward Galois theory. 
Purely inseparable Galois theories try to establish a correspondence between certain subextensions of a purely inseparable field extension and certain subobjects of an algebraic object. For instance, Jacobson's Galois theory (\cite{Jacobson37}, \cite{Jacobson44}) establishes a one-to-one correspondence between intermediate fields of a finite purely inseparable field extension of exponent one $K\subset L$ and restricted Lie subalgebras of $\Der_K(L)$. 
 Different authors attempted to generalize this correspondence to finite purely inseparable field extensions of arbitrary finite exponent. For example, in \cite{Sweedler68}, \cite{GZ70} and \cite{Chase71}
 a Galois correspondence for the modular subextensions $M\subset L$ of a modular extension $K\subset L$
 was obtained by using the Hasse-Schmidt derivations. Recently,
 a Galois correspondence for all subextensions $M\subset L$ of an arbitrary finite purely inseparable field extension $K\subset L$ was proposed in \cite{BraWal23} by using partition Lie algebroids, obtaining a complete generalization of
 Jacobson's Galois theory. 

However there is another 
Galois theory that applies for any finite field extension, namely the Jacobson-Bourbaki correspondence. It establishes a one-to-one correspondence between the intermediate fields of a finite field extension $K\subset L$ and the unital $K$-subalgebras of the algebra of linear endomorphisms $\End_K(L)$ that are also left $L$-vector subspaces (\cite[I, \S 2, Theorem 2]{Jacobson64}).  This correspondence has been widely extended to the non-commutative setting (see for instance \cite{FengSheng}, \cite{Hochschild49}, \cite{Jacobson47}).
Since we are dealing with ring extensions rather than just field extensions, it is natural to ask if the Jacobson-Bourbaki correspondence can be generalized to the setting of commutative ring extensions. In fact, in these notes, we obtain a generalization of this theorem, which applies in particular to purely inseparable ring extensions.

\begin{thmx}\label{jacobson-bourbaki introduction}
Let $A\subset C$ be a finite extension such that $C$ is a projective $A$-module and the induced map $\Spec(C)\to \Spec(A)$ is a homeomorphism. 
Then there is a one-to-one correspondence
\begin{align*}
\left\{  \begin{array}{cc}  \text{intermediate subrings}\  A\subset B\subset C:\\ \text{$C$ is a projective $B$-module}\end{array}\right\}
\longleftrightarrow\left\{ \begin{array}{cc} \text{unital $A$-subalgebras $H\subset \End_A(C)$}:\\ \text{$H$ is a $C$-direct summand of $\End_A(C)$} \end{array}  \right\}
\end{align*}
given by the inverse maps
\begin{align*}
    B&\mapsto\End_B(C)\\
    B_H&\mapsfrom  H
\end{align*}
where    $B_H:=\{x\in C: \varphi(xc)=x\cdot \varphi(c),\ \forall \varphi\in H, \ \forall c\in C\}$.
\end{thmx}

\medskip
The hypothesis that $\Spec(A)$ and $\Spec(C)$ are homeomorphic cannot be omitted, as we show in Example \ref{ex: homeomorphism is necessary}.

\medskip 

In the case of field extensions, given a purely inseparable extension $K\subset L$, it is clear that for any intermediate subfield $K\subset E\subset L$ the subextensions $K\subset E$ and $E\subset L$ are also purely inseparable. The situation is much more complicated for purely inseparable ring extensions. If $A\subset C$ is a purely inseparable extension of rings and $B$ is an intermediate subring,  i.e., $A\subset B\subset C$, it is not true in general that $A\subset B$ and $B\subset C$ are also purely inseparable. Indeed, as we will see, we have to impose some additional conditions to the subring $B$ to reach this situation.
For the exponent one case, this problem is already well understood. Given a Galois extension $A\subset C$ and a subring $B$ such that $A\subset B\subset C$, André proves that $B\subset C$ is Galois if and only if $C$ is a projective $B$-module (\cite[Théorème 71]{Andre91}). If this occurs, it follows that $A\subset B$ is also a Galois extension. 
In contrast, if $A\subset C$ is a purely inseparable ring extension of exponent greater than one and $B$ is an intermediate field such that $C$ is projective as $B$-module, then it is not necessarily true that $A\subset B$ and $B\subset C$ are also purely inseparable. We exhibit this situation in Example \ref{example failed purely inseparable tower}.

The next goal in these notes is to find conditions for the intermediate subrings $B$ of a purely inseparable extension $A\subset C$ which ensure that both subextensions $A\subset B$ and $B\subset C$ are also purely inseparable. When this occurs, we say that $A\subset B\subset C$ is a \textit{purely inseparable tower} of rings. In order to provide conditions for a tower of rings to be purely inseparable, we introduce the notion of $\mc{F}$-extension for finite extensions of rings of finite exponent.   

\begin{definition}
    Let $A\subset C$ be a finite ring extension of finite exponent. We say that $A\subset C$ is an \textit{$\mc{F}$-extension} if $C$ is a projective $A[C^{p^e}]$-module for all $e\geq 0$. 
\end{definition}

It is clear that any purely inseparable extension is in particular an $\mc{F}$-extension. 
We shall prove the converse under very strong conditions on $B$.

\begin{thmx}\label{some conditions for a purely inseparable tower}
Let $A\subset C$ be a purely inseparable extension and $B$ an intermediate subring. Assume that
\begin{enumerate}
    \item $A\subset B$ and $B\subset C$ are $\mc{F}$-extensions, and
    \item for $1\leq e<\exp(C/B)$, each of the extensions $A\subset B[C^{p^e}]$ is an $\mc{F}$-extension.
\end{enumerate} 
 Then
$A\subset B\subset C$ is a purely inseparable tower.
\end{thmx}

We do not know if condition (2) is a consequence of (1).
We shall see that this is the case in exponent two; hence:

\begin{corx}\label{purely inseparable towers of exponent two}
Let $A\subset C$ be a purely inseparable extension of exponent two and let $B$ be an intermediate ring. Then $A\subset B\subset C$ is a purely inseparable tower if and only if $A\subset B$ and $B\subset C$ are both  $\mc{F}$-extensions.  
\end{corx}

\medskip

Finally, given two purely inseparable extensions $A\subset B$ and $B\subset C$ we consider the question of whether the extension $A\subset C$ is also purely inseparable, that is, whether $A\subset B\subset C$ is a purely inseparable tower. This is true in the case of fields. However,  Example \ref{ex: not p.i. comp of Galois extensions} shows that the result is not longer valid in the setting of rings. Our final result shows that the question has a positive answer if we impose some extra conditions.

\begin{thmx}\label{composition of purely inseparable extensions}
Let $A\subset B$ and $B\subset C$ be purely inseparable extensions. Assume that $A\subset B[C^{p^e}]$ is an $\mc{F}$-extension for all $0\leq e<\exp(C/B)$. Then $A\subset C$ is also purely inseparable.   
\end{thmx}
We do not know if in every purely inseparable tower $A\subset B\subset C$ it holds that $A\subset B[C^{p^e}]$ is an $\mc{F}$-extension for all $0\leq e<\exp(C/B)$. This is the case at least when $A\subset B$ has exponent one; see Remark \ref{rem: on the condition of F-extension}

\subsection{Organization} 
The paper is largely expository and it is organized in five sections in addition to this introduction. 
In Section \ref{Sec: preliminaries}, we collect some preliminaries on finite projective modules, derivations, differential operators and modules of principal parts. In Section \ref{Section on Galois extensions}, we review some results on purely inseparable extensions of exponent one, also known as Galois extensions of exponent one, and we provide a new characterization for these extensions in terms of derivations. 
 In Section \ref{sec: purely inseparable}, we first review the notion of purely inseparable extensions of arbitrary exponent and the existing characterizations. We then prove Theorem \ref{purely inseparable new characterization} (see Theorem \ref{thm: purely inseparable and differential operators}), which provides new characterizations using differential operators and re-obtains a preexisting one in terms of modules of principal parts. 
 In Section \ref{sec: Jacobson-Bourbaki}, we prove Theorem \ref{jacobson-bourbaki introduction} (see Theorem \ref{th: endomorphism correspondence}) that extends Jacobson-Bourbaki correspondence to a certain class of ring extensions (not necessarily of characteristic $p$) including the purely inseparable ring extensions.
 Finally, in Section \ref{section on purely inseparable towers}, we study purely inseparable towers of rings and prove in particular Theorem \ref{some conditions for a purely inseparable tower} (see Proposition \ref{prop: on purely inseparable towers with the strong condition of F-extensions}), Corollary \ref{purely inseparable towers of exponent two} (see Corollary \ref{cor: purely inseparable towers of exponent two}), and Theorem \ref{composition of purely inseparable extensions} (see Corollary \ref{cor: composition of purely inseparable}).

\section{Preliminaries}\label{Sec: preliminaries}

We collect here a list of important definitions and results that will be used frequently in these notes. 
In \S \ref{subsection: preliminaries on projective modules} we review basic properties on projective modules and some modules of homomorphisms, in particular the module of endomorphisms of a projective module. 
In \S\ref{subsection: preliminaries on derivations} we review the definition and some standard properties of derivations and the module of Kähler differentials. Finally, in \S\ref{sec: preliminaries on differential operators and principal parts} we focus on differential operators and the modules of principal parts. 

\medskip

\subsection{Finite projective modules and $\Hom$}\label{subsection: preliminaries on projective modules}
We review some relevant properties of finitely generated projective modules and the modules of homomorphisms, with a particular view toward finite ring extensions $A\subset C$, where $C$ is a projective $A$-module. 
We begin with a well-known characterization of finite projective modules. 

\begin{proposition}[{\cite[II, \S 5.2, Theorem 1 and Corollary 2]{Bourbaki} }]\label{characterization of projective modules}
    Let $A$ be a ring and let $M$ be a finitely generated $A$-module.
    Then the following conditions are equivalent:
    \begin{enumerate}
        \item $M$ is a projective $A$-module.
        \item $M$ is a flat $A$-module of finite presentation.
        \item There are $f_1,\ldots,f_r\in A$ generating the unit ideal such that $M_{f_i}$ is a free $A_{f_i}$-module for all $i$.
    \end{enumerate}
    Moreover, in this situation, if $\p\subset A$ is a prime ideal and $t_1,\ldots,t_n\in M_\p$ form a basis for $M_\p$ over $A_\p$ then there exists $f\in A\setminus \p$ such that $t_1,\ldots,t_n$ form a  basis for $M_f$ over $A_f$.
\end{proposition}

Most of the results that follow are easy consequences of the above characterization of finite projective modules. 
Nevertheless, we sometimes include references or proofs to ease the reading. 
In the first result we recall that the notion of finite projective module is preserved under base extension and faithfully flat descent.

\begin{proposition}[{\cite[I, \S 3.6, Proposition 12]{Bourbaki}}]\label{projective modules are preserved by base change}
    Let $M$ be a finite $A$-module and let $A\to A'$ be a ring extension. Let $M':=A'\otimes_A M$.
    \begin{enumerate}
        \item If $M$ is projective $A$-module, then $M'$ is a projective $A'$-module.
        \item If $M'$ is a projective $A'$-module and $A\to A'$ is faithfully flat then $M$ is a projective $A$-module.
    \end{enumerate}
\end{proposition}
\begin{proposition}\label{restriction of scalar of projective modules}
    Let $A\subset B$ be a ring extension and let $M$ be a $B$-module. If $B$ is a finite projective $A$-module and $M$ is a finite projective $B$-module, then $M$ is a finite projective $A$-module.
\end{proposition}
\begin{proof}
    By hypothesis, $M$ is isomorphic as $B$-module to a $B$-direct summand of $B^m$ for some $m$, and $B$ is isomorphic as $A$-module to an $A$-direct summand of $A^n$ for some $n$. Therefore, $M$ isomorphic as $A$-module to an $A$-direct summand of $A^{mn}$. So $M$ is a finite projective $A$-module. 
\end{proof}

Finite projective modules that are also ring extensions split in a particular way, as shown by the next proposition. 
\begin{proposition}[{\cite[Lemma 4.4]{dB-S-V}}]\label{projective finite extensions are split}
    If $A\subset C$ is a finite extension of rings such that $C$ is a projective $A$-module, then $C=A\oplus N$ for some $A$-submodule $N\subset C$. In other words, the exact sequence of $A$-modules $0\to A\to C\to C/A\to 0$ is split.
\end{proposition}

\begin{remark}
    In particular, if $A\subset C$ is a finite extension of rings such that $C$ is projective over $A$, then  for any ring extension $A\to A'$, we have an inclusion $A'\subset A'\otimes_A C$. 
\end{remark}

\begin{proposition}\label{tower of projective extensions}
Let $A\subset C$ be a ring extension such that $C$ is a finite projective $A$-module. For an intermediate ring $A\subset B\subset C$ the following conditions are equivalent.
\begin{enumerate}
    \item $C$ is a finite projective $B$-module and $B$ is a finite projective $A$-module.
    \item $C$ is a flat $B$-module and $B$ a finite $A$-module.
    \item $C$ is a projective $B$-module.
\end{enumerate}
\end{proposition}

\begin{proof}
    It is clear that $(1)$ implies $(2)$. 
    
    We now assume (2) to prove (3). By the first equivalence in Proposition \ref{characterization of projective modules}, in order to show that $C$ is projective as $B$-module we only need to show that $C$ is a $B$-module of finite presentation. Firstly, $C$ is a finite $B$-module because it is a finite $A$-module. Then there is a surjective map of $B$-modules $\varphi \colon B^m\to C$ for some $m$. Let $K$ be its kernel. 
    Note that $\varphi$ is in particular a surjective map of $A$-modules. By assumption, $B$ is a finite $A$-module, so the same is true for $B^m$. Since $C$ is a finitely presented $A$-module, $K$ must be finitely generated as $A$-module (see \cite[I, \S 2.8, Lemma 9]{Bourbaki}). In particular, $K$ is a finitely generated $B$-module. Hence $C$ is a $B$-module of finite presentation.
   
    Finally, we assume (3) and prove (1). Note that $C$ is a finite $B$-module because it is a finite $A$-module. We now show that $B$ is a finite projective $A$-module.
    By Proposition \ref{projective finite extensions are split}, $B$ is a direct summand of $C$ as $B$-module, hence it is a direct summand of $C$ as $A$-module. Since $C$ is a finite projective $A$-module, $B$ must be a finite projective $A$-module too.
\end{proof}

\medskip

 We now focus on the modules of homomorphisms. \color{black} Let $A$ be a ring and let $C$ be an $A$-algebra. If $M$ is an $A$-module and $N$ is a $C$-module, then the $A$-module $\Hom_A(M,N)$ of $A$-linear maps from $M$ to $N$ has a natural structure of left $C$-module: if $\varphi\in \Hom_A(M,N)$ and $c\in C$ then $c\cdot\varphi$ is defined as $(c\cdot\varphi) (m)=c\varphi(m)$. In addition, there is a natural isomorphism of left $C$-modules 
\begin{equation}
    \Hom_A(M,N)\cong \Hom_C(C\otimes_A M,N)
\end{equation}
described as follows: for each $A$-linear map $D\colon M\to N$ there exists a unique $C$-linear map $\varphi\colon C\otimes_A M\to N$ such that $\varphi(1\otimes m)=D(m)$. In particular, for $M=C$, we obtain an isomorphism of left $C$-modules
\begin{align*}
    \Hom_A(C,N)\cong \Hom_C(C\otimes_A C,N),
\end{align*}
where $C\otimes_A C$ is viewed as left $C$-module as follows: $c(x\otimes y)=cx\otimes y$. 

More particularly, when $M=N=C$ we obtain an isomorphism of left $C$-modules
\begin{align*}
    \End_A(C):=\Hom_A(C,C)\cong \Hom_C(C\otimes_A C,C)=: (C\otimes_A C)^*.
\end{align*}
We call $\End_A(C)$ the module of endomorphism of $A\subset C$. 
As discussed above, it will be considered as left $C$-module and sometimes also as an $A$-module.
Observe that any $c\in C$ defines an endomorphism $C\to C$, namely the one given by $x\mapsto cx$. 
In this way, we can see $C$ as a $C$-submodule of $\on{End}_A(C)$.
Note finally that $\on{End}_A(C)$ is also a non-commutative unital $A$-algebra under composition. This structure will be exploited in Section \ref{sec: Jacobson-Bourbaki}.

\medskip
We recall the fact that under suitable hypothesis on $M$, $\on{Hom}(M,N)$ behaves well under localization.  

\begin{proposition}[{\cite[I, \S 2.10, Proposition 11]{Bourbaki}}]\label{localization of Hom}
    Let $M$ be an $A$-module of finite presentation. Let $C$ be an $A$-algebra and let $N$ be any $C$-module.
    \begin{enumerate}
        \item If $S\subset A$ is a multiplicative subset, then the natural map $$S^{-1}\Hom_A(M,N)\to \Hom_{S^{-1}A}(S^{-1}M, S^{-1}N)$$ is an isomorphism of left $S^{-1}C$-modules.
        \item More generally, if $A\to A'$ is a flat homomorphism, then the natural map $$A'\otimes_A\Hom_A(M,N)\to \Hom_{A'}(A'\otimes_A M, A'\otimes_A N)$$ is an isomorphism of left $A'\otimes_A C$-modules.
    \end{enumerate}
\end{proposition}

 In particular, for a finite extension of rings $A\subset C$ where $C$ is a projective $A$-module, the module of endomorphisms behave well under localization:
\begin{corollary}\label{localization of End}
    Let $C$ be an $A$-algebra that is also an $A$-module of finite presentation. For any prime ideal $\p\subset A$ and any element $f\in A$, the natural maps $(\End_A(C))_f\to \End_{A_f}(C_f)$ and $(\End_A(C))_\p\to \End_{A_\p}(C_\p)$ are isomorphisms of $C_f$-modules and $C_\p$-modules respectively.
\end{corollary}

The following proposition is a straightforward consequence of Proposition \ref{characterization of projective modules}.
\begin{proposition}\label{dual and double dual of a projective module}
    If $M$ is a finite projective $A$-module then the dual module $M^*$ is also a finite projective $A$-module and the natural map to the double dual $M\to (M^*)^*$ is an isomorphism of $A$-modules.
\end{proposition}

\begin{corollary}\label{cor: End is finite and projective}
    Let $A\subset C$ be a finite extension such that $C$ is a finite projective $A$-module. Then $\on{End}_A(C)$ is  both a finite projective $A$-module and a finite projective left $C$-module.
\end{corollary}
\begin{proof}
    By Proposition \ref{projective modules are preserved by base change}, $C\otimes_A C$ is projective left $C$-module, and by Proposition \ref{dual and double dual of a projective module}, $(C\otimes_A C)^*$ is again a finite projective $C$-module.
    Finally, as we noted before, there is an isomorphism of left $C$-modules $\End_A(C)\simeq (C\otimes_A C)^*$. This proves that $\on{End}_A(C)$ is a finite projective $C$-module. It is also a finite projective $A$-module by Proposition \ref{restriction of scalar of projective modules}.
\end{proof}

Finally, we prove a proposition that makes use of almost all the properties stated above. This proposition will be used twice in these notes.

\begin{proposition}\label{existence of special basis for a direct summand of End}
Let $A\subset C$ be a finite extension such that $C$ is a projective $A$-module and the induced map $\Spec(C)\to\Spec(A)$ is a homeomorphism. Let $H\subset \on{End}_A(C)$ be a $C$-module direct summand.
Given a prime ideal $\p\subset A$, there are $f\in A\setminus \p$,  $t_1,\ldots,t_n\in C$,  and 
$\phi_1,\ldots,\phi_l\in H_f$, for some $l\leq n$,   such that
\begin{enumerate}
    \item $C_f$ is a free $A_f$-module and $t_1,\ldots,t_n$ induce a basis for this module.
    \item $H_f$ is a free $C_f$-module and $\phi_1,\ldots,\phi_l$ is a basis for this module.
    \item $\phi_i(t_j)=\delta_{ij}$ in $C_f$ for $1\leq i,j \leq l$.
\end{enumerate}  
\end{proposition}
\begin{proof}
Since $C$ is a finite and projective $A$-module, $C\otimes_A C$ is a finite projective left $C$-module. 
We identify the left $C$-module $\End_A(C)$ with the dual $(C\otimes_A C)^*=\Hom_C(C\otimes_A C,C)$ as we did before. 
By Proposition \ref{dual and double dual of a projective module}, there is a natural identification $C\otimes_A C\cong ((C\otimes_A C)^*)^*$.
Hence, in the identification $\End_A(C)\cong (C\otimes_A C)^*$, the $C$-module direct summand $H$ corresponds to $\Hom_C(P,C)\subset \on{Hom}_C(C\otimes_A C,C)$ for some quotient $C\otimes_A C\to P$, where $P$ is a projective $C$-module.

The hypothesis implies that $C_\p$ is a free $A_\p$-module of finite rank, say $n$. 
Choose elements $t_1,\ldots,t_n\in C$ that induce a basis for $C_\p $ as $A_\p $-module. 
Then $1\otimes t_1,\ldots,1\otimes t_n\in C\otimes_A C$ induces a basis for $C_\p\otimes_{A_\p} C_\p $ as left $C_\p $-module.
The hypothesis of homeomorphic spectra implies that $C_\p$ is also a local ring, so the projective quotient $P_\p$ of $C_\p\otimes_{A_\p} C_\p$ is also a free $C_\p$-module, say of rank $l\leq n$. Again, since $C_\p$ is local, after rearrangement we may assume that the images of $1\otimes t_{1},\ldots,1\otimes t_l$ in $P$ induce a basis for $P_\p$ as $C_\p$-module. Let $\phi_1,\ldots,\phi_l\in \on{Hom}_{C_\p}(P_\p,C_\p)$ be the dual basis.
When viewing these as elements of $\on{Hom}_{C_\p}(C_\p\otimes_{A_\p} C_\p,C_\p)$ we have $\phi_i(1\otimes t_j)=\delta_{ij}$ in $C_\p$ for all $1\leq i,j\leq l$.

Coming back to $H\subset \on{End}_A(C)$, we have found a basis $\phi_1,\ldots,\phi_l$ for the left $C_\p$-module $H_\p\subset \on{End}_{A_\p}(C_\p)$ and elements $t_1,\ldots,t_n\in C$ such that $\phi_i(t_j)=\delta_{ij}$ in $C_\p$. Since $H$ is a finite projective $C$-module and $A\subset C$ induces a homeomorphism of spectra, there is $f\in A\setminus \p$ such that $\phi_1,\ldots,\phi_l$ are in $H_f\subset \on{End}_{A_f}(C_f)$ and they form a basis for $H_f$ as $C_f$-module. After multiplying $f$ by some other element of $A\setminus \p$ we may even assume that $\phi_i(t_j)=\delta_{ij}$ holds in $C_f$ for all $1\leq i,j\leq l$. This finishes the proof.
\end{proof}

\medskip

\subsection{Derivations and module of K\"ahler differentials}\label{subsection: preliminaries on derivations}
We review the definition and some basic properties of the modules of derivations and the module of K\"ahler differentials. See \cite{Kunz86} for details.

\medskip

Let $C$ be an $A$-algebra and let $M$ be a $C$-module. A homomorphism of $A$-modules $C\to M$ is called an \textit{$A$-derivation} if it satisfies $D(c_1c_2)=c_1D(c_2)+c_2D(c_1)$ for all $c_1,c_2\in C$. The set of $A$-derivations is denoted by $\Der_A(C,M)$. It is a $C$-submodule of $\Hom_A(C,M)$. We remark that the construction of $\Der_A(C,M)$ is functorial in $M$.

We now recall the construction of the universal derivation. Let $J$ be the kernel of the multiplication map $C\otimes_A C\to C$ and let $\Omega_{C/A}=J/J^2$. We consider it with the structure of $C$-module deduced from the structure of left $C$-module in $C\otimes_A C$. The $C$-module $\Omega_{C/A}$ is 
called the module of \textit{K\"ahler differentials}. The map $d\colon C\to\Omega_{C/A}$ associating to each $c\in C$ the class of $1\otimes c-c\otimes 1$ is called the \textit{universal $A$-derivation}. In fact, it is an $A$-derivation and satisfies the following universal property: the map $\varphi\mapsto \varphi\circ d$ gives an isomorphism of $C$-modules 
\begin{align}\label{Derivations as dual of Omega}
\Hom_C(\Omega_{C/A},M)\cong \Der_A(C,M).
\end{align}

\begin{proposition}[{\cite[Proposition 4.21 and Corollary 4.22]{Kunz86}}]\label{omega and base change}
    Let $C$ be an $A$-algebra. If $A\to A'$ is any homomorphism of rings and $C':=A'\otimes_AC$, then there is a natural isomorphism
    \begin{align*}
        \Omega_{C'/A'}\cong C'\otimes_C\Omega_{C/A}.
    \end{align*}
    In particular, if $S\subset A$ a multiplicative subset, then there is a natural isomorphism $\Omega_{S^{-1}C/S^{-1}A}=S^{-1}\Omega_{C/A}$.
\end{proposition} 

\begin{proposition}[{\cite[Propositions 4.12 and 4.17]{Kunz86}}]\label{Omega of a finitely presented algebra}
    Let $C$ be an $A$-algebra of finite presentation. Then $\Omega_{C/A}$ is a $C$-module of finite presentation.
\end{proposition}
As an immediate consequence of Proposition \ref{Omega of a finitely presented algebra}, the isomorphism (\ref{Derivations as dual of Omega}) and Proposition \ref{localization of Hom}, we obtain a good localization property for the modules of derivations.
\begin{corollary}\label{derivations and localization}
    If $C$ is an $A$-algebra of finite presentation and $M$ is a $C$-module, then the natural map $S^{-1}\Der_A(C,M)\to \Der_{S^{-1}A}(S^{-1}C,S^{-1}M)$ is an isomorphism.
\end{corollary}

Finally, in the particular case in which $M=C$, $\Der_A(C,C)$ is denoted simply as $\Der_A(C)$. It is a left $C$-submodule of $\End_A(C)$. It is not an $A$-subalgebra of $\on{End}_A(C)$ as the  composition of derivations is not a derivation in general.
Given $D,D'\in \End_A(C)$, we instead consider the bracket $$[D,D']:=D\circ D'-D'\circ D.$$
This bracket is $A$-bilinear and makes $\End_A(C)$ into a Lie $A$-algebra. It holds that $\Der_A(C)$ is indeed a Lie subalgebra of $\End_A(C)$.
In other words, if $D,D'\in \Der_A(C)$, then $[D,D']\in \Der_A(C)$ too. In characteristic $p>0$, it also holds that if $D\in\Der_A(C)$, then the $p$-times composition $D^p$ is also in $\Der_A(C)$. In this case, $\Der_A(C)$ is a restricted Lie algebra.

\medskip

\subsection{Differential operators and modules of principal parts}\label{sec: preliminaries on differential operators and principal parts}
In this last part, we review the definition and some basic properties of differential operators and the modules of principal parts. 

\medskip

Let $C$ be an $A$-algebra  and let $M$ be a $C$-module. For any $A$-linear map $D\colon C\to M$ and for any $x\in C$, we define an $A$-linear map $[x,D]\colon C \to M$ by the formula $$[x,D](c):=xD(c)-D(xc) \quad \text{for each $c\in C$.}$$ 
    In other words, if we consider $\Hom_A(C,M)$ with both its natural structures of left $C$-module and right $C$-module, and if for each $x\in C$ we denote by $L_x, R_x\colon \Hom_A(C,M)\to \Hom_A(C,M)$ the left and right multiplications by $x$, respectively, then
    \begin{align*}
        [x,D]:=(L_x-R_x)(D).
    \end{align*}
    We say that $D\colon C\to M$ is an \textit{$A$-differential operator} from $C$ to $M$ if there exists a non-negative integer $n$ such that 
    \begin{equation}\label{condition of differential operator}
    [x_0[x_1[\dots,[x_n,D]\dots]]]=0 \quad \text{ in $\Hom_A(C,M)$ for all $x_0,\dots,x_{n}\in C$.}
    \end{equation}
    The smallest integer $n$ with this property is called {\em the order of $D$}. 

    \medskip 
    
    For example, an $A$-differential operator of order 0 is just a $C$-linear map $D\colon C\to M$, so it is uniquely determined by $D(1)\in M$. Any $A$-derivation $D\colon C\to M$ is a differential operator of order $\leq 1$, and any $A$-differential operator of order $\leq 1$ can be written uniquely as a sum of a differential operator of order 0 and an $A$-derivation.

    \medskip 
    
    We denote by $\Diff_A^n(C,M)$ the set of $A$-diferential operators from $C$ to $M$ of order $\leq n$, and
    by $\Diff_A(C,M)$ the set of all $A$-differential operators.  These are $C$-submodules of the left $C$-module $\on{Hom}_A(C,M)$ and they are functorial in $M$ too. 
    There is a filtration
    \begin{align}\label{eq: filtration of Hom}
        \Diff_A^0(C,M)\subset \Diff_A^1(C,M)\subset\Diff_A^2(C,M)\subset\cdots\subset \Hom_A(C,M)
    \end{align}
    of the left $C$-module $\Hom_A(C,M)$, which may not be exhaustive in general. As remarked previously, there is a natural identification of $C$-modules
    $\Diff_A^0(C,M)=M$ given by $D\mapsto D(1)$, and a decomposition into a direct sum of $C$-submodules 
    $$\Diff_A^1(C,M)= M\oplus \Der_A(C,M),\quad D=D(1)+(D-D(1)).$$
    
     When $M=C$, we simply write $\Diff_A(C)=\Diff_A(C,C)$ and $\Diff_A^n(C)=\Diff_A^n(C,C)$ for all $n\geq 0$. 
    
    \medskip
    
    We now review the construction of the universal $A$-differential operator of order $n$. 
We consider $C\otimes_A C$ as a $C$-algebra via the map $c\mapsto c\otimes 1$.
    Let $J$ be the kernel of the multiplication map $C\otimes_A C\to C$. For each non-negative integer $k$,
    the quotient $C$-module
    $$\on{P}_{C/A}^k:=(C\otimes_A C)/J^{k+1}$$
    is called the \textit{module of principal parts} of order $k$ of $C$ relative to $A$.
The map
$$\delta_k \colon C\to \on{P}_{C/A}^k, \quad x\mapsto\overline{ 1\otimes x},$$ 
is called the \textit{universal $A$-differential operator} of order $k$. In fact, it is an $A$-differential operator of order $k$ satisfying the following universal property: for any $C$-module $M$, the map
\begin{align}\label{dual of principal parts}
        \Hom_C(\on{P}^k_{C/A},M)\to \Diff_A^k(C,M), \quad \varphi\mapsto \varphi \circ \delta_k
    \end{align}
    is an isomorphism of $C$-modules
    (see \cite[Proposition 16.8.4]{EGAIV}).

\medskip

We state two basic properties of these modules.

\begin{proposition}[ {\cite[Proposition 16.4.5]{EGAIV}}]\label{prop: principal parts and base change}
Let $C$ be an $A$-algebra and $k\geq 0$. If $A\to A'$ is any homomorphism of rings and $C':=C\otimes_A A'$, then 
\begin{align}\label{eq: base change principal part}
	\on{P}_{C'/A'}^k=C'\otimes_C \on{P}_{C/A}^k
\end{align} 
In particular, if $S\subset A$ is a multiplicatively closed subset, then $S^{-1}\on{P}_{C/A}^k=\on{P}_{S^{-1}C/S^{-1}A}^k$.
\end{proposition}

\begin{proposition}[{\cite[Corollaire 16.4.22]{EGAIV}}] \label{prop: principal parts and finite presentation}
    Let $C$ be an $A$-algebra of finite presentation. Then $\on{P}_{C/A}^k$ is a $C$-module of finite presentation.
\end{proposition}
As an immediate consequence of Proposition \ref{prop: principal parts and finite presentation}, the isomorphism (\ref{dual of principal parts}) and Proposition \ref{localization of Hom}, we obtain a good localization property for the modules of differential operators.

\begin{corollary}\label{cor: localization of diff}
    Let $C$ be an $A$-algebra of finite presentation, let $M$ be a $C$-module and let $k\geq 0$. If $S\subset A$ is a multiplicative subset, then there is a canonical identification
    \begin{align*}
        S^{-1}\Diff_A^k(C,M)=\Diff_{S^{-1}A}^k(S^{-1} C, S^{-1} M).
    \end{align*}
\end{corollary}

\section{Purely inseparable extensions of exponent one}\label{Section on Galois extensions}

Purely inseparable ring extensions of exponent one were studied in depth by Yuan (\cite{Yuan70/1}), who refers to them as \textit{Galois extensions of exponent one}. In this section, we revisit some basic properties 
and characterizations of Galois extensions. In addition, we give a new characterization for Galois extensions using the module of derivations.  Our guideline for this section is Theorem \ref{thm: Galois characterizations}, which collects these characterizations of Galois extensions and introduces a new one.

Specifically, in \S \ref{subsection: p-basis} we review the concept and some basic properties of $p$-basis, as well as some existence criteria. In \S \ref{subsection: Galois extension} we use $p$-basis to
introduce Galois extensions, as Yuan does in \cite{Yuan70/1}, and explain why this notion coincides with our notion of purely inseparable extension of exponent one. We also review the characterization of Galois extensions in terms of the module of K\"ahler differentials.
In \S \ref{subsection: Galois extensions and derivations} we review the condition of being a Galois extension in terms of derivations and endomorphisms. Proposition \ref{properties of Der(C) for extensions with p-basis} collects two necessary conditions for a finitely presented flat extension $A\subset C$ of exponent one to be Galois: (1) $\Der_A(C)$ and $C$ generate $\End_A(C)$ as $A$-algebra, and (2) $\Der_A(C)$ is a $C$-direct summand of the left $C$-module $\End_A(C)$. Yuan proved that  condition (1) is also sufficient, while we prove here the sufficiency of (2).
Finally, in \S \ref{sec: Galois extensions} we collect  all the above mentioned characterizations of Galois extensions. 

\medskip

The section revisits many results included in Yuan's paper \cite{Yuan70/1}. In addition, we recall some results that can be also found in Kunz's book \cite{Kunz86}. Proofs of propositions are included here when they are short to facilitate the reading. The contents from this section will be used later when we consider purely inseparable extensions of arbitrary exponent.  We recall that a finite extension of finite exponent $A\subset C$ is purely inseparable if and only if $A[C^{p^{e+1}}]\subset A[C^{p^e}]$ is a Galois extension for all $0\leq e < \exp(C/A)$ (Theorem \ref{thm: characterization of purely inseparable}).

\medskip

\subsection{On $p$-basis}\label{subsection: p-basis} 
We recall the definition and some elementary properties of $p$-basis for ring extensions. 
The notion of $p$-basis will be used in \S \ref{subsection: Galois extension} to introduce Galois extensions of exponent one. 
 As we will see, these are extensions of exponent one that admit locally a finite $p$-basis.

\begin{definition}\label{def: p-basis}
	Let $A\subset C$ be a ring extension of exponent one and let $x_1,\dots,x_n\in C$. We say that $\{x_1,\dots,x_n\}$ is a \textit{$p$-basis} for $C$ over $A$ if 
	\begin{enumerate}[(i)]
		\item $x_1,\dots,x_n$ generate $C$ as $A$-algebra, i.e, $C=A[x_1,\dots,x_n]$, and
		\item $\{x_1,\dots,x_n\}$ is a \textit{$p$-independent set} over $A$, i.e., the set of monomials $\{x_1^{\alpha_1}\cdot\ldots\cdot x_n^{\alpha_n} : 0\leq \alpha_i < p\}$ is linearly independent over $A$. 
	\end{enumerate}
	In other words, $\{x_1,\dots, x_n\}$ is a $p$-basis if $C$ is a free $A$-module with basis the set of \textit{reduced monomials} $\{\x^\alpha:=x_1^{\alpha_1}\cdot\ldots\cdot x_n^{\alpha_n}: 0\leq \alpha_i<p\}$, or equivalently, if the $A$-algebra homomorphism $A[X_1,\ldots,X_n]\to C$ which maps $X_i$ to $x_i$ induces an isomorphism
	\begin{align*}
		A[X_1,\ldots,X_n]/\langle X_1^p-x_1^p,\ldots,X_n^p-x_n^p\rangle\cong C.
	\end{align*}
 \end{definition} 

\begin{example}
It is well-known that every finite field extension of exponent one admits a $p$-basis. More generally, every exponent-one finite extension of regular local rings has a $p$-basis. This was proved by Kimura and Nitsuma in \cite{KN82}. 
\end{example}

The following lemma collects some well-known (and easy to prove) properties of $p$-basis.
\begin{lemma}\label{basic properties of p-basis}
	Let $A\subset C$ be a ring extension of exponent one.
	\begin{enumerate}
		\item If $\{x_1,\ldots,x_n\}$ is a $p$-basis for $C$ over $A$, then for any $a_1,\ldots,a_n\in A$, $\{x_1-a_1,\ldots,x_n-a_n\}$ is also a $p$-basis for $C$ over $A$.
		\item If $\{x_1,\ldots,x_n\}$ is a $p$-basis for $C$ over $A$ and $A'$ is an $A$-algebra, then $\{x_1\otimes 1,\ldots, x_n\otimes 1\}\subset C\otimes_A A'$ is a $p$-basis for $C\otimes_A A'$ over $A'$. 
		\item Let $B$ be an intermediate ring, $A\subset B\subset C$. If $\{x_1,\ldots,x_r\}$ is a $p$-basis for $C$ over $B$ and $\{x_{r+1},\ldots,x_n\}$ is a $p$-basis for $B$ over $A$, then $\{x_1,\ldots,x_n\}$ is a $p$-basis for $C$ over $A$.
	\end{enumerate}
\end{lemma}

We will be particularly interested in some $p$-basis, called \textit{splitting $p$-basis}, that induce an special presentation of the ring extension.

\begin{definition}
    Let $A\subset C$ be a ring extension of exponent one. 
    A {\em splitting $p$-basis} for $C$ over $A$ is a $p$-basis $\{x_1,\ldots,x_n\}$ such that $x_i^p=0$ for all $i$. In other words, it is a $p$-basis that induces an $A$-algebra isomorphism $C\cong A[X_1,\ldots,X_n]/\langle X_1^p,\ldots,X_n^p\rangle$. 
\end{definition}

Splitting $p$-basis can be obtained from arbitrary $p$-basis by allowing a suitable base change if necessary. More specifically, 
let $A\subset C$ be a ring extension of exponent one that admits a $p$-basis $\{x_1,\ldots,x_n\}$.
	Let  $A\subset A'$ be any ring extension such that, for each $i=1,\ldots,n$, there exists $y_i\in A'$ with $y_i^p=x_i^p$.
 By Lemma \ref{basic properties of p-basis}(2), $\{x_1\otimes 1,\ldots,x_n\otimes 1\}$ is a $p$-basis for $C\otimes_A A'$ over $A'$.
 Let $z_i:=x_i\otimes 1-1\otimes y_i\in  C\otimes_A C$. 
 By Lemma \ref{basic properties of p-basis}(1),  $\{z_1,\ldots,z_n\}$ is also a $p$-basis of $ C\otimes_A A'$ over $A'$. 
 It is in fact a splitting $p$-basis since $z_i^p=x_i^p\otimes 1-1\otimes y_i^p=1\otimes (x_i^p-y_i^p)=0$. Note that we can take $A'=C$, or just $A'=A$ if $A=A^p$.  Thus,

\begin{corollary}\label{coro: p-basis becomes splitting p-basis}
    Let $A\subset C$ be a ring extension of exponent one that admits a $p$-basis of $n$ elements.
\begin{enumerate}
    \item If $A=A^p$ (e.g., $A$ a perfect field), 
    then $C\cong A[X_1,\ldots,X_n]/\langle X_1^p,\ldots,X_n^p\rangle$ as $A$-algebras.
    \item $C\otimes_A C\cong C[X_1,\ldots,X_n]/\langle X_1^p,\ldots,X_n^p\rangle$ as (left) $C$-algebras.
    \end{enumerate}
\end{corollary}

\medskip

We review now  the notion of {\em differential basis} in relation to the notion of $p$-basis. Indeed, for finite extensions of exponent one, the concepts of $p$-basis and differential basis coincide.

\begin{definition}
    Let $C$ be a finitely generated $A$-algebra and let $x_1,\ldots,x_n\in C$. We say that $\{x_1,\ldots,x_n\}$ is a \textit{differential basis} if the $C$-module $\Omega_{C/A}$ is free and  $\{dx_1,\ldots,dx_n\}$ is a basis.
\end{definition}

\begin{proposition}\label{p-basis=differential basis}
	Let $A\subset C$ be a finite extension of exponent one and let $x_1,\ldots,x_n\in C$. Then $\{x_1,\ldots,x_n\}$ is a $p$-basis of $C$ over $A$ if and only if it is a differential basis.
\end{proposition}
The direct implication is easy and can be found in \cite[Proposition 5.6]{Kunz86}. The difficult part in the proof of the converse is to show that if
$x_1,\ldots,x_n$ form a differential basis then they generate $C$ as $A$-algebra. The complete proof can be found in \cite[Proposition 3.6]{dB-S-V}. We include in the following lemma an argument used in that proof. This lemma will be used later in some arguments.
\begin{lemma}\label{generating sets and omega}
    Let $A\subset C$ be a finite extension of rings of finite exponent and let $x_1,\ldots,x_n\in C$. Then $C=A[x_1,\ldots,x_n]$ if and only if $dx_1,\ldots,dx_n$ generate $\Omega_{C/A}$ as $C$-module.
\end{lemma}
\begin{proof}
    The $S$-module $\Omega_{C/A}$ is generated by $\{dc\colon c\in C\}$. In particular, if $C=A[x_1,\dots,x_n]$, by an application of the Leibniz rule, it follows that $dx_1,\dots,dx_n$ generate $\Omega_{C/A}$ as $C$-module.
	
	Now, we assume that $dx_1,\dots,dx_n$ generate $\Omega_{C/A}$ as $C$-module. Let $C'=A[x_1,\dots,x_n]$. Note that $dx_1,\dots,dx_n$ also generate $\Omega_{C'/A}$. Thus, the natural map $C\otimes_{C'}\Omega_{C'/A}\to \Omega_{C/A}$ is surjective. As we have an exact sequence
	$$C\otimes_{C'}\Omega_{C'/A}\to \Omega_{C/A}\to \Omega_{C/C'}\to 0,$$
	we deduce that $\Omega_{C/C'}=0$. Since $C'\subset C$ is a finite extension of finite exponent, it follows from \cite[Lemma 3.5]{dB-S-V}
 that $C=C'$. Hence, $x_1,\dots,x_n$ generate $C$ as $A$-algebra. 
\end{proof}
 \color{black}
 
As an immediate consequence of Proposition \ref{p-basis=differential basis} and Lemma \ref{generating sets and omega}
we obtain some criteria for the existence of $p$-basis in the local case.
\begin{corollary}\label{free omega implies existence of p-basis}
    Let $A\subset C$ be an exponent one finite extension of local rings. Then $C$ admits a $p$-basis over $A$ if and only if $\Omega_{C/A}$ is a flat $C$-module. If this is the case, any minimal generating set $x_1,\ldots,x_n$ of $C$ as an $A$-algebra is a $p$-basis.
\end{corollary}

\begin{proof}
By Proposition \ref{p-basis=differential basis}, we only need to prove
that if $\Omega_{C/A}$ is a free $C$-module and $x_1,\ldots,x_n$ is a minimal generating set for the $A$-algebra $C$,
then it is a differential basis.
Now, by Lemma \ref{generating sets and omega}, $dx_1,\ldots,dx_n$ is a minimal generating set for $\Omega_{C/A}$, so it has to be a basis since $C$ is a local ring.   
\end{proof}

\begin{corollary}\label{cor: descent of p-basis}
    Let $A\subset C$ be an exponent one finite extension of local rings. Assume that there is a faithfully flat homomorphism $A\to A'$ such that $C\otimes_A A'$ has $p$-basis over $A'$.
    Then $C$ has a $p$-basis over $A$.
\end{corollary}
\begin{proof}
Let $C':=C\otimes_A A'$. Since $A\to A'$ is faithfully flat, the base change $C\to C'$ is so too.
Since $C'$ has $p$-basis over $A'$,  $\Omega_{C'/A'}$ is a flat $C'$-module by Proposition \ref{p-basis=differential basis}. This module is the base change of the $C$-module $\Omega_{C/A}$ under the faithfully flat map $C\to C'$. By faithfully flat descent, $\Omega_{C/A}$ is a flat $C$-module.
Thus, by Corollary \ref{free omega implies existence of p-basis}, we conclude that $A\subset C$ has a $p$-basis.
\end{proof}

\subsection{Galois extensions of exponent one}\label{subsection: Galois extension}
We recall Yuan's notion of Galois extension for ring extensions of exponent one.  Later, we will show that this notion of Galois extension coincides with our notion of purely inseparable extension of exponent one.

\begin{definition}\label{def: Galois extension}
    Let $A\subset C$ be a finite ring extension of exponent one. 
    We say that $A\subset C$ is a \textit{Galois extension of exponent one} if 
    \begin{enumerate}[(i)]
    \item  $C$ is a projective $A$-module, and
    \item for each prime ideal $\p \subset A$, the extension $A_\p\subset C_\p:=C\otimes_A{A}_\p$ admits a $p$-basis.
    \end{enumerate}
\end{definition}

\begin{remark}  By Proposition \ref{characterization of projective modules}, condition (i) in the above definition can be replaced by the condition that $C$ is an $A$-module of finite presentation, since (ii) already implies that $C_\p$ is a flat $A_\p$-module for all prime ideal $\p\subset A$.
\end{remark}
The  following lemma, which is a direct consequence of Definition \ref{def: Galois extension} and Lemma \ref{basic properties of p-basis}, tells us that Galois extensions behave well under base extension and under composition when restricted to exponent one extensions.

\begin{lemma}\label{lem: basics on Galois extensions}
    Let $A\subset C$ be a finite ring extension of exponent one.
    \begin{enumerate}
        \item If $A\subset C$ has a $p$-basis, then $A\subset C$ is a Galois extension. 
        \item Let $A\to A'$ be any ring homomorphism. If $A\subset C$ is a Galois extension, then $A'\subset C\otimes_A A'$ is also a Galois extension.
        \item Let $B$ be any intermediate subring $A\subset B\subset C$. If $A\subset B$ and $B\subset C$ are Galois extensions, then $A\subset C$ is also a Galois extension.
    \end{enumerate}
\end{lemma}

We now explain why the notion of Galois extension coincides with the notion of purely inseparable extension of exponent one.
We recall the latter from Definition \ref{def: purely inseparable} in the introduction.

\begin{definition}
    We say that $A\subset C$ is a \textit{purely inseparable extension of exponent one} if
    \begin{enumerate}[(i)]
    \item $A\subset C$ has exponent one,
    \item $C$ is a finite projective $A$-module, and
    \item for every  prime ideal $\p\subset A$, there exists a faithfully flat homomorphism $A_\p\to A'$ such $C_\p\otimes_{A_\p} A'$ has a splitting $p$-basis over $A'$.
    \end{enumerate}
\end{definition}

\begin{proposition}\label{prop: Galois= purely inseparable}
  Let $A\subset C$ be a finite ring extension of exponent one. Then $A\subset C$ is a Galois extension if and only if it is a purely inseparable extension.
\end{proposition}
\begin{proof}
   This follows from Corollary \ref{coro: p-basis becomes splitting p-basis}(2) and Corollary \ref{cor: descent of p-basis} applied to $A_\p\subset C_\p$  for each prime ideal $\p\subset A$. 
\end{proof}

The following proposition includes a characterization of Galois extensions in terms of the module of K\"ahler differentials.
\begin{proposition}\label{prop: Galois iff omega projective}
    Let $A\subset C$ be a finite ring extension of exponent one. Then the following conditions are equivalent.
    \begin{enumerate}
        \item $A\subset C$ is a Galois extension.
        \item $C$ is an $A$-module of finite presentation and $\Omega_{C/A}$ is a flat $C$-module.
        \item $\Omega_{C/A}$ is a projective $C$-module.
        \item There exist $f_1,\ldots,f_r\in A$ generating the unit ideal such that $C_{f_i}$ has a $p$-basis over $A_{f_i}$.
    \end{enumerate}
\end{proposition}
\begin{proof}
The implication (1)$\Rightarrow$(2) follows from the formula $(\Omega_{C/A})_\p=\Omega_{C_\p/A_\p}$ for all prime ideal $\p\subset A$ (Proposition \ref{omega and base change}), 
and Corollary \ref{free omega implies existence of p-basis}. 
The implication (2)$\Rightarrow$(3) follows from the fact that if $C$ is an $A$-module of finite presentation, then $\Omega_{C/A}$ is a $C$-module of finite presentation (Proposition \ref{Omega of a finitely presented algebra}).

We assume now that $\Omega_{C/A}$ is projective as $C$-module. To prove (3)$\Rightarrow$(4),
we only need to show that given a prime ideal $\p\subset A$, there is $h\in A\setminus \p$ such that $C_h$ has a $p$-basis over $A_h$.
Since $\Spec(A)$ and $\Spec(C)$ are homeomorphic and $\Omega_{C/A}$ is a finite projective $C$-module, there is $f\in A\setminus\p$ such that $\Omega_{C_f/A_f}$ is a finite free $C_f$-module, say of rank $n$. Choose $x_1,\ldots,x_n\in C$ such that when they are viewed as elements of $C_\p$ then $dx_1,\ldots,dx_n$ is a minimal generating set for the $C_\p$-module $\Omega_{C_\p/A_\p}$. 
Then there is $g\in A\setminus \p$ such that $\Omega_{C_{fg}/A_{fg}}$ is generated by $dx_1,\ldots,dx_n$ (we view here $x_1,\ldots,x_n$ as elements in $C_{fg}$). 
Since $\Omega_{C_{fg}/A_{fg}}$ is a free $C_{fg}$-module of rank $n$,  $dx_1,\ldots,dx_n$ has to be a basis for this $C_{fg}$-module. By Proposition \ref{p-basis=differential basis}, $x_1,\ldots,x_n$ is a $p$-basis for $C_{fg}$ over $A_{fg}$. This ends the proof (3)$\Rightarrow$(4).
Finally, the implication (4)$\Rightarrow$(1) is trivial.
\end{proof}

\medskip

\subsection{Galois extensions and derivations}\label{subsection: Galois extensions and derivations}
We explore the connection between Galois extensions and their derivations viewed as endomorphisms of the extension. More precisely, let $A\subset C$ be a ring extension of exponent one such that $C$ is a finite projective $A$-module. We consider the left $C$-module of $A$-endomorphisms $\End_A(C)$. This  is a finite and projective $C$-module as shown in Corollary \ref{cor: End is finite and projective}. 
It can be also considered as a unital $A$-algebra under composition. 
In \cite{Yuan70/1}, Yuan shows that $A\subset C$ is Galois if and only if $\End_A(C)$ is generated as an $A$-algebra by $C$ and the 
$A$-derivations,  i.e., $\End_A(C)=C[\Der_A(C)]$. 
Here, we consider $\Der_A(C)$ as $C$-submodule of $\End_A(C)$ and
show that $A\subset C$ is Galois if and only if $\Der_A(C)$ is a $C$-module direct summand of $\End_A(C)$. 
Later, in 
\S \ref{sec: purely inseparable and differential operators}, we will extend this result and give a characterization of purely inseparable extensions of higher exponent in terms of modules of differential operators viewed as $C$-submodules of $\End_A(C)$.

\begin{parrafo}
\label{derivations associated with a p-basis}	\textit{Derivations associated to a $p$-basis.} Let $A\subset C$ be a ring extension of exponent one that admits a $p$-basis $\{x_1,\dots, x_n\}$. Denote $ \x^{\alpha}:=x_1^{\alpha_1}\cdot\dots \cdot x_n^{\alpha_n}$, for each $\alpha=(\alpha_1,\ldots,\alpha_n)\in \mb{N}_0^n$, and let $\mc{A}:=\{\alpha=(\alpha_1,\ldots,\alpha_n)\in\mb{N}_0^n:\ 0\leq \alpha_i<p\}$. Then the set of reduded monomials $\{\x^\alpha:\ \alpha\in\mc{A}\}$ is a basis for $C$ over $A$. For each $i\in\{1,\dots,n\}$, we denote by $\e_i$ the $n$-tuple such that $(\e_i)_j=\delta_{ij}$. Then for each $i$ there is an $A$-linear map $\d_i \colon C\to C$ satisfying
$$\d_i(\x^{\alpha})=\begin{cases}
		\alpha_i\cdot \x^{\alpha-\e_i} & \text{if $\alpha_i\neq 0$}\\
		0 & \text{if $\alpha_i=0$.}
	\end{cases}$$
	One easily shows that  $\d_1,\dots,\d_n$ are in fact $A$-derivations.
 We refer to them as the \textit{partial derivatives associated to the $p$-basis} $\{x_1,\dots,x_n\}$. Every other $A$-derivation $\d\colon C\to C$ can be written uniquely as a $C$-linear combination of the partial derivatives, namely
 \begin{align*}
     \d=\d(x_1)\d_1+\cdots+\d(x_n)\d_n.
 \end{align*}
Thus, $\{\d_1,\dots,\d_n\}$ is a basis for $\Der_A(C)$ as $C$-module.

More generally, for each $\beta\in\mc{A}$ there is an $A$-linear map $\d_\beta \colon C\to C$ given by the formula
\begin{align*}
\d_\beta(\x^\alpha)=\binom{\alpha}{\beta}\x^{\alpha-\beta},   
\end{align*}
where the binomial coeffiecient is defined as
$$\binom{\alpha}{\beta}:=\prod_{i=1}^n \binom{\alpha_i}{\beta_i}.$$
Each $\d_\beta$ is a differential operator of order $|\beta|:=\beta_1+\cdots+\beta_n$, and the whole collection $\{\d_\beta:\ \beta\in\mc{A}\}$ is a basis for the $C$-module $\End_A(C)$.
Moreover, we observe that
\begin{align*}
\d_\beta:=\frac{1}{\beta!}\d_1^{\beta_1}\circ\cdots\circ\d_n^{\beta_n}, \qquad   \text{where $\beta!:=\prod_{i=1}^n\beta_i!$.}
\end{align*}

These observations combined with Corollary \ref{derivations and localization} imply the following.
\end{parrafo}
\begin{proposition}
    [{\cite[Theorem 9]{Yuan70/1}}]\label{properties of Der(C) for extensions with p-basis}
     Let $A\subset C$ be a finite ring extension of exponent one such that $C$ is a projective $A$-module. If $A\subset C$ be a Galois extension, then
    \begin{enumerate}
        \item $\End_A(C)=C[\Der_A(C)]$,
        \item $\Der_A(C)\subset \End_A(C)$ is a $C$-module direct summand.
    \end{enumerate}
    In addition, $A=\{x\in C: \partial(x)=0,\ \forall \partial\in\Der_A(C)\}$.
\end{proposition}
Yuan proved that condition (1) in the above proposition is in fact sufficient for $A\subset C$ to be Galois.

\begin{proposition}[{\cite[Theorem 10]{Yuan70/1}}] \label{prop: Der generates End implies Galois}
    Let $A\subset C$ be a finite ring extension of exponent one such that $C$ is a projective $A$-module.
    If $\End_A(C)=C[\Der_A(C)]$, then $A\subset C$ is a Galois extension.
\end{proposition}

Our goal is to demonstrate that condition (2) in Proposition \ref{properties of Der(C) for extensions with p-basis} is also sufficient for $A\subset C$ to be Galois. We prove this in Proposition \ref{prop: direct summand of endomorphism implies Galois}.
For the proof  we shall need the following three technical lemmas. 

\begin{lemma}\label{commuting derivations and p-basis}
Let $C$ be a ring and let $x_1,\ldots,x_n\in C$.
If there exist derivations $\d_1,\dots, \d_n$ such that $$\d_i(x_j)=\delta_{ij},\quad [\d_i,\d_j]=0, \quad \d_i^p=0, \quad \forall i,j\in\{1,\dots,n\},$$
then $x_1,\ldots,x_n$ is a $p$-basis for $C$ over $A:=\cap_{i=1}^n\ker(\d_i)$ and  $\{\d_1,\dots, \d_n\}$ are the associated partial derivatives.  In particular, $\{\d_1,\dots, \d_n\}$
is a basis for the $C$-module $\Der_A(C)$.
\end{lemma}

A proof of this fact is included in the last paragraph of the proof of \cite[Theorem 12]{Yuan70/1}, which in turn follows the lines of Hochschild's proof of \cite[Theorem 1]{Hochschild55}. We include a proof here for the sake of convenience for the reader.
    
\begin{proof}
     For $\beta=(\beta_1,\ldots,\beta)\in\mb{N}_0^n$ we set $D_\beta:=\d_1^{\beta_1}\circ\cdots \circ \d_n^{\beta_n}$. One easily shows that $\frac{1}{\beta!}D_\beta(\x^\alpha)=\binom{\alpha}{\beta}\x^{\alpha-\beta}$ for all $\alpha,\beta\in\mc{A}:=\{\alpha=(\alpha_1,\ldots,\alpha_n):\ 0\leq\alpha_i<p\}$, and $D_\beta=0$ if $\beta\notin\mc{A}$ by the condition $\partial_i^p=0$ for all $i$.
     Using the first observation, one easily proves that $\{x_1,\ldots,x_n\}$ is a $p$-basis for $C':=A[x_1,\ldots,x_n]$ over $A$. It remains to show that $C'=C$.
     
    For a non-zero $c\in C$, we prove that $c\in C'$ by induction on $d(c):=\max\{|\alpha|: \alpha\in\mc{A}, D_\alpha(c)\neq 0\}$, which exists because $D_0(c)=c\neq 0$ and $\mc{A}$ is a finite set.
    If $d(c)=0$, then $\d_i(c)=0$ for all $i$, and therefore $c\in A\subset C'$. 
    Assume that $d(c)=d>0$.
    For $\alpha\in\mc{A}$ with $|\alpha|=d$, we set $a_\alpha=\frac{1}{\alpha!}D_{\alpha}(c)$. Given $i=1,\ldots,n$, note that $\partial_i(a_\alpha)=\frac{1}{\alpha!}D_{\alpha+\e_i}(c)$. This is zero if $\alpha+\e_i\notin \mc{A}$ because $D_{\alpha+\e_i}=0$ in this case, and it is also zero if  $\alpha+\e_i\in \mc{A}$ by the maximal property of $d$.
    It follows that $a_\alpha\in A$, by definition of $A$.
    Let $c'=c-\sum_{|\alpha|=d} a_\alpha \x^\alpha$. Then either $c'=0$ or else $d(c')<d$. Thus $c'\in C'$ by the inductive hypothesis and so $c\in C'$. This completes the proof of the fact that $\{x_1,\ldots,x_n\}$ is a $p$-basis for $C$ over $A$. Note that $\d_1,\ldots,\d_n$ must be the associated partial derivatives, and so they from a basis for the $C$-module $\on{Der}_A(C)$, as discussed in \ref{derivations associated with a p-basis}.
\end{proof}

\begin{lemma}\label{ring extension with no derivations}
    Let $A\subset C$ be a finite ring extension of exponent one such that $C$ is projective as $A$-module. If $\Der_A(C)=0$, then $A=C$.
\end{lemma}

\begin{proof}
    The (left) $C$-module $C\otimes_A C$ is projective of finite rank, since $C$ is so over $A$. 
    Let $I\subset C\otimes_A C$ be the kernel of the multiplication map. Then $C\otimes_A C=C\oplus I$ and so $I$ is also a projective $C$-module of finite rank. 
    The fact that $A\subset C$ is finite also ensures that $I$ is a finitely generated ideal, and the fact that $A\subset C$ has finite exponent implies that $I$ is nilpotent.  We prove by induction that $\Hom_C(I^k/I^{k+1},C)=0$ for all $k\geq 1$. The case $k=1$ is the hypothesis $\on{Der}_A(C)=0$. Let $k>1$ and suppose that $\Hom_C(I^{k-1}/I^k,C)= 0$. If there exists a non-zero $\varphi\in\Hom_C(I^k/I^{k+1},C)$, then $\varphi(xy+I^{k+1})\neq 0$ for some $x\in I$ and  $y\in I^{k-1}$. But then the map $\psi(u+I^k):=\varphi(xu+I^{k+1})$ is a non-zero $C$-module morphism $I^{k-1}/I^k\to C$, contradicting the inductive hypothesis. Thus $\Hom_C(I^k/I^{k+1},C)=0$. 
    Now choose $n$ such that $I^{n+1}=0$. Then $\Hom_C(I,C)=\on{Hom}_C(I/I^{n+1},C)$ and by the above claim this is equal to $\Hom_C(I/I^{n-1},C)=\cdots=\Hom_C(I/I^2,C)=0$. Therefore, $\Hom_C(I,C)=0$ and, since $I$ is a projective $C$-module, we conclude that $I=0$. Thus $C\otimes_A C=C$, and since $A\to C$ is faithfully flat we deduce that $C=A$.
\end{proof}

\begin{lemma}\label{lem: equal derivations implies equal rings}
    Let $A\subset C$ be a finite ring extension of exponent one such that $C$ is a projective $A$-module. 
    Let $B$ be an intermediate ring, $A\subset B\subset C$, such that  $C$ has a $p$-basis over $B$ and $\Der_A(C)=\Der_B(C)$. Then $A=B$.
\end{lemma}
\begin{proof}
The hypothesis implies that $B$ is a finite projective $A$-module. 
Let  $\{x_1,\ldots,x_n\}$ be a $p$-basis for $C$ over $B$. Let $a_i=x_i^p\in A$. Then $$C=\frac{B[X_1,\ldots,X_n]}{\langle X_1^p-a_1,\ldots, X_n^p-a_n\rangle}=B\otimes_A\frac{A[X_1,\ldots,X_n]}{\langle X_1^p-a_1,\ldots,X_n^p-a_n\rangle}.$$
    If $A\neq B$, by Lemma \ref{ring extension with no derivations}, there is a non-zero $A$-derivation $\d \colon B\to B$, and by the above equality, there is a non-zero $A$-derivation of $C$ that extends $\d$. This contradicts $\on{Der}_A(C)=\on{Der}_B(C)$. Thus $A=B$.
\end{proof}


\begin{proposition}\label{prop: direct summand of endomorphism implies Galois}
Let $A\subset C$ be a finite ring extension of exponent one such that $C$ is a projective $A$-module. If $\Der_A(C)$ is $C$-module direct summand of $\on{End}_A(C)$,
then $A\subset C$ is a Galois extension.
\end{proposition}

\begin{proof}
By Corollary \ref{derivations and localization} and Corollary \ref{localization of End} we may assume that $A$ (and hence $C$) is local.

By Proposition \ref{existence of special basis for a direct summand of End} with $H=\Der_A(C)$, there are elements $x_1,\dots,x_n\in C$
and a basis $\d_1,\dots,\d_n$ for the $C$-module $\Der_A(C)$ such that $\d_i(x_j)=\delta_{ij}$ for all $i,j$. We claim that these elements are such that $[\d_i,\d_j]=\d_i^p=0$ for all $i,j\in\{1,\dots,n\}$.

In fact, since  $[\d_i,\d_j],\d_i^p\in \Der_A(C)$ for all $1\leq i,j\leq n$, and since
$\d_1,\ldots,\d_n$ is a basis for the $C$-module $\Der_A(C)$, we have $$[\d_i,\d_j]=\sum_{k=1}^n c_{ij}^{(k)}\d_k, \quad \text{and} \quad \d_i^p=\sum_{k=1}^n c_{i}^{(k)}\d_k \quad \text{for some $c_{ij}^{(k)},c_i^{(k)}\in C$}.$$
We note that $c_{ij}^{(k)}=[\d_i,\d_j](x_k)=0$ and $c_i^{(k)}=\d_i^p(x_k)=0$, for $k=1,\dots,n$. So $[\d_i,\d_j]=\d_i^p=0$ for all $i,j\in\{1,\dots,n\}$. 

We have shown that $x_1,\ldots,x_n$ and $\d_1,\ldots,\d_n$ satisfy the conditions in Lemma \ref{commuting derivations and p-basis}. Hence, $x_1,\dots, x_n$ is a $p$-basis for $C$ over $B:=\cap_{i=1}^n\ker(\d_i)$, and  $\d_1,\ldots,\d_n$ is a basis for the $C$-module $\on{Der}_B(C)$. 
In particular, $\on{Der}_A(C)=\on{Der}_B(C)$. We can now apply Lemma \ref{lem: equal derivations implies equal rings} and conclude that $B=A$. Thus $C$ has a $p$-basis over $A$ and $A\subset C$ is Galois.

\end{proof}

\subsection{A summary of equivalences of purely inseparable extensions of exponent one}\label{sec: Galois extensions} To sum up, we collect in the following theorem the equivalences of the notion of Galois extension of exponent one that we mentioned or proved throughout the section.

\begin{theorem}
    \label{thm: Galois characterizations}
    Let $A\subset C$ be a finite extension of exponent one. Then the following are equivalent:
    \begin{enumerate}
        \item $A\subset C$ is a purely inseparable extension.
        \item $A\subset C$ is a Galois extension.
        \item $\Omega_{C/A}$ is a projective $C$-module. 
        \item $C$ is a projective $A$-module and $\End_A(C)=C[\Der_A(C)]$.
        \item $C$ is a projective $A$-module and $\Der_A(C)$ is a $C$-direct summand of $\End_A(C)$.
    \end{enumerate}
\end{theorem}
Roughly, the aim for the next section will be to find a suitable version of this theorem for ring extensions of higher exponent.

\section{Purely inseparable extensions of higher exponent}\label{sec: purely inseparable}

Purely inseparable ring extensions of arbitrary exponent were firstly treated by Pauer in \cite{Pauer78}, where they are introduced as \textit{special algebras}.  He presented this definition as a ring extension that satisfies a series of equivalent conditions, extending in some sense the equivalences $(1)\Leftrightarrow (2)\Leftrightarrow (3)$ of Theorem \ref{thm: Galois characterizations}
to higher exponent extensions.
In \cite{PSan99}, P. S. de Salas  gave another characterization of special algebras involving the modules of principal parts, providing another extension of the equivalence (1)$\Leftrightarrow$(3) of Theorem \ref{thm: Galois characterizations}
to higher exponent extensions.
In this section, we review Pauer's equivalent conditions that define purely inseparable extensions and we re-obtain P. S. de Salas's characterization from a different point of view.
We also provide a new  characterization in terms of the modules of differential operators. More specifically, we formulate and prove an extension of the equivalence (1)$\Leftrightarrow$(5) in Theorem \ref{thm: Galois characterizations} to the setting of higher exponent extensions.

\medskip

\subsection{On purely inseparable extensions of rings}
We choose one of the equivalent conditions that define Pauer's notion of special algebra, and we take it as our definition of purely inseparable extension. We then review some basic properties and characterizations of these extensions.

\begin{definition}\label{def: purely inseparable 2}
A finite ring extension $A\subset C$ 
is called {\em purely inseparable} if 
\begin{enumerate}[(i)]
    \item $A\subset C$ has finite exponent,
    \item $C$ is a (finite) projective $A$-module, and 
    \item for each prime ideal $\p\subset A$ there is a faithfully flat extension $A_\p\to A'$ such that $C_\p\otimes_{A_\p}A'$ is isomorphic to an $A'$-algebra of the form $A'[X_1,\ldots,X_n]/\langle X_1^{p^{e_1}},\ldots,X_n^{p^{e_n}}\rangle$.
\end{enumerate}
\end{definition}

\begin{remark}
Note that condition (i) in the above definition is redundant because it can be deduced from (ii) and (iii).  By Proposition \ref{characterization of projective modules},  condition (ii) can be replaced by the condition that $C$ is an $A$-module of finite presentation as condition (iii) already implies that $C_\p$ is a flat $A_\p$-module for all prime ideal $\p\subset A$. Finally, we note that  condition (iii) only needs to be verified at maximal ideals.
\end{remark}

It is clear that the three conditions in Definition \ref{def: purely inseparable 2} behave well under base extension and  faithfully flat descent. Therefore:

\begin{proposition}\label{base change of flat descent of purely inseparable}
    Let $A\subset C$ be a finite ring extension. Let $A\to A'$ be a ring map and set $C':=C\otimes_A A'$.
    \begin{enumerate}
        \item If $A\subset C$ is purely inseparable, then so is $A'\to C'$.
        \item If $A\to A'$ is faithfully flat and $A'\subset C'$ is purely inseparable, then so is $A\subset C$.
    \end{enumerate}
\end{proposition}

\medskip

    Let $C$ be an $A$-algebra. 
  For each integer $e\geq 0$, we set 
  $$C^{[e]}:=A[C^{p^e}].$$ That is, $C^{[e]}$ is the $A$-subalgebra of $C$ generated by all elements of the form $x^{p^{e}}$ for $x\in C$.
  Note that if $C$ is finitely generated over $A$, then so is each $C^{[e]}$. 
  In addition, if $A\subset C$ is an extension of finite exponent, then the extension $C^{[e+1]}\subset C^{[e]}$ is finite of exponent one, for each $e <\exp(C/A)$, and $C^{[e]}=A$ for all $e\geq \exp(C/A)$. Since $\Omega_{C^{[e]}/A}=\Omega_{C^{[e]}/C^{[e+1]}}$, Theorem \ref{thm: Galois characterizations} implies the following.

    \begin{proposition}
      Let $A\subset C$ be a finite ring extension of finite exponent. Then the following conditions are equivalent:
      \begin{enumerate}
          \item $\Omega_{C^{[e]}/A}$ is a projective $C^{[e]}$-module for all $0\leq e <\exp(C/A)$
          \item $C^{[e+1]}\subset C^{[e]}$ is a Galois extension for all  $0\leq e <\exp(C/A)$
      \end{enumerate}
  \end{proposition}

  We shall review the fact that for extensions of finite exponent these conditions characterize purely inseparable extensions. The following example might be insightful.
  \begin{example}\label{example trivial modular extension}
    Let $A$ be a ring and let $C=A[x_1,\ldots,x_n]$ be the $A$-algebra with defining relations $x_1^{p^{e_1}}=\cdots=x_n^{p^{e_n}}=0$, for certain integers $e_1\geq\cdots\geq e_n\geq 1$. Then $A\subset C$ is a purely inseparable extension of exponent $\exp(C/A)=e_1$. 
    For each integer $1\leq e\leq e_1$, we denote by $n(e)$ the largest integer $m\leq n$ such that $x_{m}^{p^e}\neq 0$. Then
    \begin{align*}
        C^{[e]}=A[x_1^{p^e},\ldots,x_{n(e)}^{p^e}]
    \end{align*}
    and the defining relations for this $A$-algebra are $(x_1^{p^e})^{p^{e_1-e}}=\ldots=(x_{n(e)}^{p^e})^{p^{e_{n(e)}-e}}=0$. It follows easily that $x_1^{p^e},\ldots,x_{n(e)}^{p^e}$ is a $p$-basis for $C^{[e]}$ over $C^{[e+1]}$. 
\end{example}

  The following proposition, whose verification is straightforward, tells us that under certain circunstances the tower $C\supset C^{[1]}\supset C^{[2]}\supset\cdots\supset A$ behaves well under base extensions.

\begin{proposition}\label{frobenius filtration and base change}
 Let $A\subset C$ be a finite extension of finite exponent and let $A\to A'$ be a ring map.
 Suppose that either $A\to A'$ is flat, or
 $C^{[e+1]}\to C^{[e]}$ is flat for all $e$.
 Then $(C\otimes_A A')^{[e]}=C^{[e]}\otimes_A A'$ for all $e\geq 0$. 
 \end{proposition}
 
In particular, for each prime ideal $\p\subset A$, we have that $(C_\p)^{[e]}=(C^{[e]})_\p$, so we can just write $C_\p^{[e]}$ without confusion.

\begin{proposition}
Let $A\subset C$ be a finite ring extension of finite exponent.
If $A\subset C$ is a purely inseparable extension, then $C^{[e+1]}\subset C^{[e]}$ is a Galois extension for all  $e < \exp(C/A)$.     
\end{proposition}
\begin{proof}
    We use an inductive argument on the exponent of the extension. The case of exponent one is Proposition \ref{prop: Galois= purely inseparable}. So we assume that $\exp(C/A)>1$. Then 
    the proposition follows if  we prove that $C^{[1]}\subset C$ is Galois and $A\subset C^{[1]}$ is purely inseparable.
    
    Let $\p\subset A$ be a prime ideal.  We first show that $A_\p\subset C_\p^{[1]}$ is purely inseparable and that $C_\p$ has a $p$-basis over $C_\p^{[1]}$.
        Since $A\subset C$ is purely inseparable, there is a faithfully flat homomorphism $A_\p\to A'$ such that $C':=C_\p\otimes_{A_\p} A'$ is of the form $A'[x_1,\ldots,x_n]$ with relations $x_1^{p^{e_1}}=\cdots=x_n^{p^{e_n}}=0$ for certain $e_1\geq\cdots\geq e_n\geq 1$. 
    By Proposition \ref{frobenius filtration and base change}, we have
    \begin{align*} {C'}^{[1]}&=C_\p^{[1]}\otimes_{A_\p} A',
    \end{align*}
    and so  using Example \ref{example trivial modular extension}, we obtain that
    $A_\p\subset C_\p^{[1]}$ is purely inseparable, and also, with the help of Corollary \ref{cor: descent of p-basis}, that $C_\p$ has a $p$-basis over $C_\p^{[1]}$.

    To conclude that $A\subset C^{[1]}$ is purely inseparable and $C^{[1]}\subset C$ is Galois, it is only left to show that $C^{[1]}$ is a projective $A$-module and that $C$ is a projective $C^{[1]}$-module.  We already know that $C$ is a finite projective $A$-module, since it is purely inseparable. We also know that $A\subset C^{[1]}$ is finite. The fact that each $C_\p$ has a $p$-basis over $C_\p^{[1]}$ also implies that $C$ is a flat $C^{[1]}$-module. Hence our desired conclusion follows from Proposition \ref{tower of projective extensions}.
\end{proof}

The following definition might be seen as a possible generalization of the notion of $p$-basis to higher exponent extensions.  It was firstly introduced by Rasala for purely inseparable field extensions (\cite{Ras71}) and then it was generalized by Pauer to the setting of finite ring extensions of finite exponent. 

\begin{definition} (\cite[Definition 2.1.5]{Pauer78}) 
 Let $A\subset C$ be a finite extension of finite exponent. We say that $x_1,\ldots,x_n\in C$ is a {\em normal generating sequence}  (NGS) for $C$ over $A$  if there are positive integers $e_1\geq \cdots\geq e_n$ and polynomials $P_i\in A[X_1,\ldots,X_{i-1}]$ for $i=1,\ldots,n$, such that  $C$ admits a presentation $C=A[x_1,\ldots,x_n]$  whose defining relations are $x_i^{p^{e_i}}=P_i(x_1^{p^{e_i}},\ldots, x_{i-1}^{p^{e_i}})$ for $i=1,\ldots,n$. 
\end{definition}

Given a finite ring extension $A\subset C$ of finite exponent, we shall prove that if the extension $C^{[e+1]}\subset C^{[e]}$ 
is Galois for all $e< \exp(C/A)$, then, locally at any prime, there is always a NGS for $C$ over $A$. With this aim, 
we review a more general construction in the setting of local ring extensions.

\medskip

  Let $A\subset C$ be a finite extension of local rings with  finite exponent.
  We denote by $n(e)$ the minimum number of generators of $C^{[e]}$ as $A$-algebra, or equivalently, as an $C^{[e+1]}$-algebra. This number coincides with the minimum number of generators of the $C^{[e]}$-module $\Omega_{C^{[e]}/A}=\Omega_{C^{[e]}/C^{[e+1]}}$ by Lemma \ref{generating sets and omega}. We write $n(0)=n$. There are inequalities
$$ n=n(0)\geq n(1)\geq n(2)\geq\cdots.$$
Moreover, observe that $n(e)=0$ if $e\geq \exp(C/A)$, since $C^{[e]}=A$ in those cases. Note also that $n(e)>0$, if $e<\exp(C/A)$. 

\begin{definition}
 In the setting described above, a sequence $x_1,\ldots,x_n\in C$ is called a \textit{generalized normal generating sequence} (GNGS) for $C$ over $A$ if $x_1^{p^e},\dots,x_{n(e)}^{p^e}$ is a minimal generating set for the $A$-algebra $C^{[e]}$ for all $e<\exp(C/A)$. 
\end{definition}
The existence of a GNGS for finite extensions of local rings with finite exponent  is clear. Indeed, let $x_1,\ldots,x_n$ be any minimal generating set of $C$ as $A$-algebra. If $\exp(C/A)=1$, then we are done. If $\exp(C/A)>1$, then $x_1^p,\ldots,x_n^p$ generate $C^{[1]}$ as $A$-algebra. After rearrangement, we may assume that the first $x_1^p,\ldots,x_{n(1)}^p$  form a minimal generating set. So, if $\exp(C/A)=2$, then we are done. Otherwise we continue in this way. The process stops at some point since $A\subset C$ has finite exponent.

Now for each $i=1,\ldots,n$ we define $e(i)$ as the smallest integer $e$ such that $n(e)<i$. In other words,
\begin{align}\label{eq: def of e(i)}
n(e(i))<i\leq n(e(i)-1).
\end{align}
 It is clear that
 $\exp(C/A)=e(1)\geq e(2)\geq\cdots\geq e(n)\geq 1.$
We call $(e(1),\ldots,e(n))$ the \textit{sequence of exponents} of the extension. Note that the sequence $n(0)\geq n(1)\geq \cdots$ can be recovered from this sequence via the formula
\begin{align*}
n(e)=\max\{ i: e(i)>e\}.
\end{align*}
Moreover, using these observations, we can easily prove that
\begin{align}\label{eq: sum of exponents}
    n(0)+n(1)+\cdots= e(1)+e(2)+\cdots+e(n).
\end{align}

\medskip

Let $x_1,\ldots,x_n\in C$  be GNGS for $C$ over $A$. Then $e(i)$ is the smallest integer $e$ such that $x_i^{p^e}\in A[x_1^{p^e},\ldots, x_{i-1}^{p^e}]=A[x_1^{p^e},\ldots,x_{n(e)}^{p^e}]$. Thus we have relations
\begin{align*}
		x_i^{p^{e(i)}}=P_i\left(x_1^{p^{e(i)}},\ldots,x_{n(e(i))}^{p^{e(i)}}\right),\quad i=1,\ldots,n
	\end{align*}
	for some polynomials $P_i\in A[Y_1,\ldots,Y_{n(e(i))}]$. Note that  we may assume that the degree of $P_i$ in the variable $Y_j$ is lower than $p^{e(j)-e(i)}$. From the above relations for $i=1,\ldots,n$, we obtain a surjective homomorphism of $A$-algebras
	\begin{align}\label{GNGS surjective morphism}
		A[X_1,\ldots,X_n]/\left\langle X_i^{p^{e(i)}}-
		P_i\left(X_1^{p^{e(i)}},\ldots,X_{n(e(i))}^{p^{e(i)}}\right): i=1,\ldots,n \right\rangle\to C
	\end{align}

\begin{proposition}
     Let $A\subset C$ be a finite extension of local rings with finite exponent and let $x_1,\ldots,x_n$ be a GNGS for $C$ over $A$. If the extension $C^{[e+1]}\subset C^{[e]}$ has a $p$-basis for each $0\leq e < \exp(C/A)$, then the map in (\ref{GNGS surjective morphism}) is an isomorphism, and therefore $x_1,\ldots,x_n$ is a NGS for $C$ over $A$. 
\end{proposition}
\begin{proof}
    It is easy to show that the ring on the left of (\ref{GNGS surjective morphism}) has rank $p^{e(1)+\cdots+e(n)}$ as $A$-module. 
    On the other hand, $x_1^{p^{e}},\ldots,x_{n(e)}^{p^e}$ is a $p$-basis for $C^{[e+1]}\subset C^{[e]}$ by Corollary \ref{free omega implies existence of p-basis}. Thus, $C$ is a free $A$-module of rank $p^{n(1)+\cdots+n(\exp(C/A)-1)}$, which is equal to $p^{e(1)+\cdots+e(n)}$ by (\ref{eq: sum of exponents}). Thus (\ref{GNGS surjective morphism}) is a surjective morphism between free $A$-modules of the same rank, so it must be an isomorphism.
\end{proof}

\begin{corollary}
    Let $A\subset C$ be a finite extension of finite exponent. If $C^{[e+1]}\subset C^{[e]}$ is a Galois extension for each $0\leq e < \exp(C/A)$, then for any prime ideal $\p\subset A$, the extension $A_\p\subset C_\p$ has a NGS.
\end{corollary}

Finally, we show that the existence of a NGS locally at any prime ideal implies that the extension is purely inseparable.
\begin{proposition}
    Let $A\subset C$ be a finite ring extension of finite exponent such that $C$ is a projective $A$-module. If the extension $A_\p\subset C_\p$ has a NGS for each prime ideal $\p\subset A$,
    then $A\subset C$ is purely inseparable.
\end{proposition}
\begin{proof}
    Let $\p\subset A$ be a prime ideal and let $x_1,\ldots,x_n\in C_\p$ be a NGS for $C_{\p}$ over $A_{\p}$. Then there are positive integers $e_1\geq\cdots\geq e_n$ and polynomials $P_i\in A_\p[X_1,\ldots,X_{i-1}]$ for $i=1,\ldots,n$, such that $C_\p=A_\p[x_1,\ldots,x_n]$ where the defining relations are $x_i^{p^{e_i}}=P_i(x_1^{p^{e_i}},\ldots, x_{i-1}^{p^{e_i}})$. 
    Let us write $$P_i(X_1,\ldots,X_{i-1})=\sum_{\alpha\in\mc{A}_i} a_{i,\alpha}X_1^{\alpha_1}\cdots X_{i-1}^{\alpha_{i-1}},$$ where $\mc{A}_i$ is a finite set of $i-1$-tuples. We define the $A_{\p}$-algebra 
    \begin{align*}
        A':=A_\p[z_{i,\alpha}: \alpha\in \mc{A}_i, i=1,\ldots,n]
    \end{align*}
    with relations $ z_{i,\alpha}^{p^{e_i}}=a_{i,\alpha}$.
    Note that $A'$ is a faithfully flat $A_\p$-algebra, and that the $A'$-algebra $C':=C_\p\otimes_{A_\p} A'$  has a presentation $C'=A_\p[x_1,\ldots,x_n]$ with relations
    \begin{align*}
    x_i^{p^{e_i}}=\sum_{\alpha\in \mc{A}_i} z_{i,\alpha}^{p^{e_i}} (x_1^{p^{e_i}})^{\alpha_1}\cdots (x_{i-1}^{p^{e_i}})^{\alpha_{i-1}}\,.
    \end{align*}
    We finally define $y_i:=x_i-\sum b_{i,\alpha} x_1^{\alpha_1}\cdots x_{i-1}^{\alpha_{i-1}}$. Then $C'=A'[y_1,\ldots,y_n]$, and the defining relations are exactly $y_i^{p^{e_i}}=0$. Thus $A\subset C$ is purely inseparable.
\end{proof}

To sum up, we have  proved the following equivalences for a finite ring extension of finite exponent to be purely inseparable. 

\begin{theorem}[Satz 2.2.5, \cite{Pauer78}]\label{thm: characterization of purely inseparable}
    Let $A\subset C$ be a finite ring extension of finite exponent. Then the following conditions are equivalent.
    \begin{enumerate}
        \item $A\subset C$ is purely inseparable.
        \item $C^{[e+1]}\subset C^{[e]}$ is a Galois extension for all $0\leq e <\exp(C/A)$. 
        \item $\Omega_{C^{[e]}/A}$ is a projective $C^{[e]}$-module for all  $0\leq e <\exp(C/A)$. 
        \item $C$ is a projective $A$-module and for each prime ideal $\p\subset A$, the $A_\p$-algebra $C_\p$ has a NGS.
    \end{enumerate}
\end{theorem}

\begin{corollary}\label{cor: purely inseparable inductively}
    Let $A\subset C$ be a finite ring extension of finite exponent. Then $A\subset C$ is purely inseparable if and only if $A\subset C^{[1]}$ is purely inseparable and $C^{[1]}\subset C$ is Galois.
\end{corollary}

\medskip 

We end this review with the natural question of whether a flat family of purely inseparable extensions defines a purely inseparable extension. 
In the case of exponent one,
a finite extension $A\subset C$ is Galois if and only if $C$ is a projective $A$-module and for each maximal ideal $\m\subset A$ the extension $A/\m\subset C/\m C$ is Galois, or equivalently, it has $p$-basis. This follows easily from Nakayama's Lemma. 
However, Pauer observed that this characterization is not longer true for purely inseparable extensions of higher exponent. 
The result would be true if for every maximal ideal $\m\subset A$  we could ensure that $C^{[e]}\otimes_A A/\m=(C/\m C)^{[e]}$. 
Now, according to Proposition \ref{frobenius filtration and base change},
this condition is satisfied by the following class of extensions.
\begin{definition}\label{def: F-extension}
    Let $A\subset C$ be a finite extension of finite exponent. We say that $A\subset C$ is an \textit{$\mc{F}$-extension} if the following equivalent conditions hold:
    \begin{enumerate}[(i)]
        \item $C^{[e]}$ is a projective $C^{[e+1]}$-module for all $e\geq 0$.
        \item $C$ is a projective $C^{[e]}$-module for all $e\geq 0$.
        \item $C^{[e]}$ is a projective $A$-module for all $e\geq 0$.
    \end{enumerate}
\end{definition}
\begin{remark}\label{rem: base change and F-extensions}
    Let $A\subset C$ be an $\mc{F}$-extension and let $A\to A'$ be any morphism of rings. Denote $C':=C\otimes_A A'$. Then  $C^{[e]}\otimes_A A'=C'^{[e]}$ for all $e\geq 0$, by Proposition \ref{frobenius filtration and base change}. Since being projective is stable under base change, it follows that $A'\subset C'$ is also an $\mc{F}$-extension.
\end{remark}

\begin{theorem}[Satz 2.2.5, \cite{Pauer78}]\label{thm: fibers of purely inseparable}
   Let $A\subset C$ be a finite extension of finite exponent. Then $A\subset C$ a is purely inseparable extension if and only if 
    \begin{enumerate}
        \item $A\subset C$ is an $\mc{F}$-extension.
        \item For any maximal ideal $\m\subset A$, the extension $A/\m A\subset C/\m C$ is purely inseparable.
    \end{enumerate}
\end{theorem}

\medskip

\subsection{Extending differential operators along a Galois extension}
We are going to prove a technical result on differential operators 
(Proposition \ref{prop: extension of differential operators}) which establishes that for a Galois extension $C'\subset C$ of $A$-algebras, a differential operator of $C'$ of order $\leq k$ can be extended (locally) to a differential operator of $C$ of order $\leq kp$.
This result will be used in
\S \ref{sec: purely inseparable and differential operators}
to re-obtain a characterization of purely inseparable extensions in terms of the modules of principal parts, as well as to give
 a new characterization in terms of some modules of differential operators.

\medskip

 The following lemma collects some basic and useful properties of the bracket for differential operators. We shall use these properties later, especially in the proof of Proposition \ref{prop: extension of differential operators} on extension of differential operators.

\begin{lemma}\label{bracket properties}
    Let $C$ be an $A$-algebra and let $M$ be a $C$-module. Let $D\colon C\to M$ be an $A$-linear map. For any $x,y\in C$ and $a\in A$, we have:
	\begin{enumerate}[(1)]
		\item[(0)] (Linearity) $[ax,D]=a[x,D]$ and $[x+y,D]=[x,D]+[y,D];$
		 
		\item  (Symmetry) $[x,[y,D]]=[y,[x,D]]$;
		
		\item (Decomposition for a product) $[xy,D]=x[y,D]+y[x,D]-[y,[x,D]];$ 
		
		\item (Reduction to generators) Let $S\subset C$ be a set and $n$ a positive integer. If $[s_0,[s_1,[\ldots,[s_n,D]\ldots]]]=0$ for any $s_0,\ldots,s_n\in S$, then $[f_0,[f_1,[\ldots,[f_n,D]\ldots]]]=0$ for any $f_0,\ldots,f_n\in A[S]$.
	
		\item (Composition $p$-times) For any $x\in C$,  $[\underbrace{x,[\cdots,[x}_p,D]\cdots]]=[x^p,D]$.
	\end{enumerate}
\end{lemma}

\begin{proof}
    (0): The linearity of the bracket is straightforward from the definition.
    \medskip
		
    (1): For any $x,y,c\in C$, we have: 
	$$[x,[y,D]](c)=x[y,D](c)-[y,D](xc)=xyD(c)-xD(yc)-yD(xc)+D(yxc).$$
	Note that this expression is symmetric in $x,y$. Thus, $[x,[y,D]]=[y,[x,D]]$ as we wanted.
	
    \medskip
	
    (2): For any $x,y,c\in C$, we have:
	\begin{align*}
		[xy,D](c)&=xyD(c)-D(xyc)=xyD(c)-xD(yc)+xD(yc)-D(xyc)\\
		&=x[y,D](c)+[x,D](yc)=x[y,D](c)+y[x,D](c)-[y,[x,D]](c).
	\end{align*}
	
	(3): By the linearity of the bracket mentioned in (0), it is enough to prove the assertion for $f_j\in A[S]$ that are monomial expressions on the elements of $S$. In this case we can proceed by induction on $\deg f_0+\cdots+\deg f_n$. If $\deg f_0=\cdots=\deg f_n=1$, then  $[f_0,[f_1,[\ldots,[f_n,D]\ldots]]]=0$ by the hypothesis. If $\deg f_i>1$ for some $i$, then by the symmetry in (1) we can assume that $\deg f_0>1$, so that $f_0=s g_0$ for some $s\in S$ and some monomial expression $g_0$ on $S$ of degree $<\deg f_0$. By applying the decomposition for a product in (2) we obtain
 \begin{align*}
     [f_0,[f_1,[\ldots[f_n,D]\ldots]]]=s[g_0,[f_1,[\ldots[f_n,D]\ldots]]]+g_0[s,[f_1,[\ldots[f_n,D]\ldots]]]-[g_0,[s,[f_1,[\ldots[f_n,D]\ldots]]]]
 \end{align*}
 The right hand side is zero by the inductive hypothesis, and so the left hand side is zero too.

    \medskip
	
    (4) Let $L_x,R_x \colon \Hom_A(C,M)\to \Hom_A(C,M)$ be as in  \S \ref{sec: preliminaries on differential operators and principal parts}. Since $L_x$ and $R_x$ commute, we have
	$$[\underbrace{x,[\cdots,[x}_p,D]\cdots]]=(L_x-R_x)^p(D)=(L_x^p-R_x^p)(D)=[x^p,D].$$ 
\end{proof}

As a first application, we show that the filtration (\ref{eq: filtration of Hom}) involving the modules of differential operators is exhaustive for finite extensions of finite exponent.
\begin{proposition}\label{differentials = homomorphisms}
Let $A\subset C$ be a finite ring extension of finite exponent, and let $M$ be a $C$-module. Then there exists $n$ such that
$\on{Diff}_A^n(C,M)=\on{Diff}_A(C,M)=\on{Hom}_A(C,M)$.    
\end{proposition}
\begin{proof}
    Let $e:=\exp(C/A)$ and let $S$ be a finite generating set for $C$ as an $A$-algebra, say with $r$ elements. Let $D\colon C\to M$ be an $A$-linear map. Set $n:=r p^e -1$. We claim that $D$ is a differential operator of order $\leq n$. By property (3) of Lemma \ref{bracket properties} , we only need to check that $[s_0,[s_1,[\ldots,[s_n,D]\ldots]]]=0$ for any $s_0,\ldots,s_n\in S$. Now, since $|S|=r$, among any $n+1=r p^e$ elements $s_0,\ldots,s_n\in S$ there must be one which appears at least $p^e$ times. By property (1) of Lemma \ref{bracket properties}, we may assume that the last $p^e$ elements in $s_0,\ldots,s_n$ are all equal to a single element, say $s$. Finally, by applying (4) of Lemma \ref{bracket properties} exactly $e$ times we obtain that $[s_0,[s_1,[\ldots,[s_{n},D]\ldots]]]=[s_0,[s_1,[\ldots,[s_{n-p^e},[s^{p^e},D]]\ldots]]]$, which is zero because $s^{p^e}\in A$ and $D$ is $A$-linear.
\end{proof}

\begin{proposition}\label{restriction of differential operator to C1}
Let $C$ be an $A$-algebra, and let $M$ be a $C$-module. Then the restriction map
\begin{align*}
    \on{res} \colon \Hom_A(C,M)\to \Hom_A(C^{[1]},M),\quad D\mapsto D|_{C^{[1]}},
\end{align*}
is a morphism of left $C$-modules, and for any $k\geq 0$, it induces a restriction map
\begin{align*}
    \on{res} \colon \Diff_A^{pk}(C,M)\to \Diff_A^{k}(C^{[1]},M).
\end{align*}
In other words, if $D\colon C\to M$ is a differential operator from $C$ to $M$ of order $\leq pk$, then its restriction $D\colon C^{[1]}\to M$ is a differential operator from $C^{[1]}$ to $M$  of order $\leq k$.   \end{proposition}
\begin{proof}
The fact that the restriction map is a morphism of $C$-modules is trivial. We prove the second part.
Let $D\in \on{Diff}_A^{pk}(C,M)$, and let $\partial$ be its restriction to $C^{[1]}$. 
Recall that $C^{[1]}=A[C^p]$, hence, by Lemma \ref{bracket properties}, in order to show that $\partial\in \on{Diff}_A^{k}(C^{[1]},M)$,  we only need to show that for any $s_0,\ldots,s_k\in C$ it holds that
    $[s_0^p,[s_1^p,[\ldots,[s_k^p,\partial]\ldots]]]=0$ in $\on{Hom}_A(C^{[1]},M)$. 
    To show this it enough to show that  $[s_0^p,[s_1^p,[\ldots,[s_k^p,D]\ldots]]]=0$ in $\on{Hom}_A(C,M)$. Now, by Lemma \ref{bracket properties}(4), $$[s_0^p,[s_1^p,[\ldots,[s_k^p,D]\ldots]]]=
    [\underbrace{s_0,[\ldots,[s_0}_{p},[\underbrace{s_1,[\ldots,[s_1}_{p},[\ldots,[\underbrace{s_k,[\ldots,[s_k}_{p},D]\ldots]]\ldots]]\ldots]]]\ldots]],$$
    which is zero because $D$ has order $\leq pk$.
\end{proof}

 Proposition \ref{prop: extension of differential operators} will imply that if $C$ has a $p$-basis over $C^{[1]}$, then the restriction map $\on{Diff}_A^{pk}(C,M)\to \on{Diff}_A^{k}(C^{[1]},M)$ is surjective. The next two lemmas will be used in the proof of that proposition.

\begin{lemma}\label{bracket development}
    Let $C$ be an $A$-algebra and let $M$ be a $C$-module. Then for any $A$-linear map  $D\colon C \to M$, any $x_1,\dots,x_n\in C$,  and  any $c\in C$, we have:
	$$[x_1,[x_2,\dots,[x_n,D]]](c)=\sum_{S\subset [n]} (-1)^{n-|S|} \x_{S} \cdot D(\x_{S^c} c),$$
	where $[n]=\{1,\dots,n\}$, $\x_S=\prod_{i\in S}x_i$ and $\x_{S^c}=\prod_{i\notin S} x_i$.
\end{lemma}

\begin{proof}
	We prove the assertion by induction on $n$. For $n=1$, the assertion is clear. Let $n>1$.
 Using that $[x_n,D]$ is an $A$-linear map from $C$ to $M$ and the inductive hypothesis for $n-1$, we obtain
	\begin{align*}
		[x_1,[x_2,\dots,[x_n,D]]](c)&=  \sum_{S\subset [n-1]} (-1)^{n-1-|S|} \x_{S} \cdot [x_n,D](\x_{S^c} \cdot c)\\
		&= \sum_{S\subset [n-1]} (-1)^{n-1-|S|} \x_{S} \cdot(x_n D(\x_{S^c} \cdot c) - D(x_n\x_{S^c} \cdot c))\\
		&= \sum_{S\subset [n-1]} (-1)^{n-1-|S|} \x_{S\cup \{n\}} \cdot D(\x_{S^c} \cdot c) + (-1)^{n-|S|}\x_{S} \cdot D(\x_{S^c\cup \{n\}} \cdot c)\\
		&= \sum_{S\subset [n]} (-1)^{n-|S|} \x_{S} \cdot D(\x_{S^c} c).
	\end{align*}
	Observe that the last equality holds because the first term of the sum includes the sets $S\subset [n]$ containing $n$ and the second term of the sum includes those that does not contain $n$. 
\end{proof}

We show now that the previous formula can be refined by grouping the repeated elements $x_i$ that are inside the bracket.
\begin{lemma}\label{grouped bracket development}
    Let $C$ be an $A$-algebra and let $M$ be a $C$-module. For any $A$-linear operator, $D\colon C \to M$ and for any $x_1,\dots,x_m\in C$, given $c\in C$, we have:
	\begin{equation*}
		\begin{split}
			&[\underbrace{x_1,[\ldots,[x_1}_{r_1},[\underbrace{x_2,[\ldots,[x_2}_{r_2},[\ldots,[\underbrace{x_m,[\ldots,[x_m}_{r_m},D]\ldots]]\ldots]]\ldots]]]\ldots]](c)\\
			&=\sum_{\substack{S_i\subset [r_i]\\ i=1,\dots,m }}\prod_{i=1}^m (-1)^{r_i-|S_i|} x_i^{|S_i|} D\left(c \cdot \prod_{i=1}^m x_i^{r_i-|S_i|} \right),
		\end{split}
	\end{equation*}
	where $r_i$ denotes the number of appearances of each $x_i$ on the bracket expression and $[d]=\{1,\dots,d\}$ for any positive integer $d$.
\end{lemma}

\begin{proof}
	Let $n=r_1+\dots + r_m$. Note that there is a partition $[n]=A_1\uplus\dots \uplus A_m$ with $|A_i|=r_i$ for $i=1,\dots,m$. If we identify $A_i$ with $[r_i]$, choosing a subset $S\subset [n]$ is equivalent to choosing subsets $S_i\subset [r_i]$ such that $|S|=|S_1|+ \dots + |S_m|$. Then by Lemma \ref{bracket development}, the first expression on the statement is 
	\begin{align*}
		\sum_{S\subset [n]} (-1)^{n-|S|} \x_{S} \cdot D(\x_{S^c} c) &=\sum_{\substack{S_i\subset [r_i]\\ i=1,\dots,m }} \prod_{i=1}^m (-1)^{r_i-|S_i|} x_i^{|S_i|} D\left(c \cdot \prod_{i=1}^m x_i^{r_i-|S_i|} \right)
	\end{align*}
	where the last equality follows from the use of the partition of $[n]$ and from observing that $\x_S=x_1^{\alpha_1}\cdot\dots \cdot x_n^{\alpha_n}$ and $\x_{S^c}=x_1^{r_1- \alpha_1}\cdot\dots \cdot x_n^{r_n - \alpha_n}$ for some $0\leq \alpha_i \leq r_i$. In fact, $\alpha_i=|S_i|$. 
\end{proof}

\begin{proposition}\label{prop: extension of differential operators}
	Let $A\subset C$ be a ring extension of finite exponent  and assume that $C$ has a $p$-basis $x_1,\ldots,x_n$ over $C^{[1]}$. Let $M$ be a $C$-module. For each $A$-linear map $\partial \colon C^{[1]}\to M$, we define an $A$-linear map $D\colon C\to M$ by the formula
	\begin{align*}
	D\left(\sum\lambda_\alpha \x^\alpha\right)=\sum \x^\alpha\partial(\lambda_\alpha) ,
	\end{align*}
	where $\lambda_\alpha\in C^{[1]}$ and the sum is taken over the reduced monomials $\x^{\alpha}:=x_1^{\alpha_1}\cdot\ldots\cdot x_n^{\alpha_n}$, with $0\leq \alpha_i<p$. 
 Then the assignment $\partial\mapsto D$ gives a morphism of left $C$-modules
 \begin{align*}
    \on{ext} \colon  \Hom_A(C^{[1]},M) \rightarrow  \Hom_A(C,M)
 \end{align*}
 such that $\on{res}\circ \on{ext}=\on{Id}$, and for each $k\geq 0$, this map induces
 \begin{align*}
    \on{ext} \colon \Diff_A^{k}(C^{[1]},M)  \rightarrow \Diff_A^{pk}(C,M)
 \end{align*}
 In other words, if $\partial$ is a differential operator of order $\leq k$, then $D$ is a differential operator of order $\leq pk$.
\end{proposition}
\begin{proof}
It is clear that $\on{ext}$ is a map of left $C$-modules such that $\on{res}\circ \on{ext}$ is the identity.
We now prove by induction on $k$ that if $\partial \colon C^{[1]}\to M$ is a differential operator of order $\leq k$, then $D \colon C\to M$ is a differential operator of order $\leq pk$. 

Let $k=0$. The $0$-order differential operator $\partial\colon C^{[1]} \to M$ is a $C^{[1]}$-linear map, so it is the
multiplication by some fixed element $m\in M$, that is, $\partial(\lambda)=\lambda m$ for all $\lambda\in C^{[1]}$. It follows from the definition of $D$ that $D(c)=cm$ for all $c\in C$. Thus, $D$ is a differential operator of order $0=p\cdot 0$, and the assertion is true.

Assume that $k>0$.
 We  have to show that $[f_0,[f_1,[\ldots[f_{pk},D]\ldots]]]]=0$ for all $f_0,\ldots,f_{pk}\in C$. Note that $C=A[x_1,\ldots,x_n]$, since $C=C^{[1]}[x_1,\dots, x_n]$ and $A\subset C$ has finite exponent.  Therefore, by Lemma \ref{bracket properties}(3), it is enough to show that 
\begin{align}\label{eq: equation to be proved}
	[x_{i_0},[x_{i_1},[\ldots[x_{i_{pk}},D]\ldots]]]]=0 \quad \forall x_{i_0},\ldots,x_{i_{pk}}\in \{x_1,\ldots,x_n\}.
\end{align}	
 
  We first enunciate a property on the evaluation of $D$ at some monomials that will be used succesively in what follows. Any monomial (not neccesarily reduced) $\x^{\beta}\in C$  admits a unique reduced expression $\x^{\beta}=\lambda \x^{\beta'}$, where $\lambda\in C^{[1]}$ and $\x^{\beta'}$ is a reduced monomial, so $D(\x^{\beta})=\partial(\lambda)\x^{\beta'}$. In particular, take $\x^{\beta}:=x_{l+1}^{\beta_{l+1}}\cdot \ldots \cdot x_n^{\beta_n}$, and let $\x^{\alpha}:=x_1^{\alpha_1}\cdot\ldots\cdot x_l^{\alpha_l}$ be a reduced monomial. Then the product $\x^{\alpha}\cdot\x^{\beta}$ is again a monomial whose reduced form is $\x^{\alpha}\x^{\beta}= \lambda \x^{\alpha} \x^{\beta'}$ and we have
	\begin{align}
		\tag{*} 
		D(\x^{\alpha}\x^{\beta})=\partial(\lambda)\x^{\alpha}\x^{\beta'}= \x^{\alpha}D(\x^{\beta})\ .
	\end{align}

\medskip

We now prove (\ref{eq: equation to be proved}).
We consider two cases according to the number of times that each element  $x_j$ of the $p$-basis appears in the list $x_{i_0},\ldots,x_{i_{pk}}$.

\medskip 

\textit{Case I: There is $j\in\{1,\dots,n\}$ such that $x_j$ appears at least $p$ times in the list $x_{i_0},\ldots,x_{i_{pk}}$.} By Lemma \ref{bracket properties}(1), we may assume that the last $p$ terms in the list $x_{i_0},\ldots,x_{i_{pk}}$ are all equal to $x_j$. Then, by Lemma \ref{bracket properties}(4), we have
\begin{align}\label{bracket case I}
[x_{i_0},[x_{i_1},[\ldots[x_{i_{pk}},D]\ldots]]]]=[x_{i_0},[x_{i_1},[\ldots[x_{i_{p(k-1)}},[x_j^p,D]]\ldots]]]]\ .
\end{align}
Note that $[x_j^p,D]$ is obtained from $[x_j^p,\partial]$ in a similar way to how $D$ is obtained from $\partial$. More precisely,
$$[x_j^p,D]\left(\sum\lambda_{\alpha}\x^{\alpha}\right)=\sum x_j^p\partial(\lambda_{\alpha})\x^{\alpha} - \partial(\lambda_{\alpha} x_j^p)\x^{\alpha}=\sum[x_j^p,\partial](\lambda_{\alpha})\cdot \x^{\alpha}.$$
Since $[x_j^p,\partial]$ is a differential operator from $C^{[1]}$ to $M$ of order $\leq k-1$, by the inductive hypothesis $[x_j^p,D]$ is a differential operator from $C$ to $M$ of order $\leq p(k-1)$. Thus the right term in (\ref{bracket case I}) is zero.

\medskip

\textit{Case II: Every $x_j$ appears at most $p-1$ times in the list $x_{i_0},\ldots,x_{i_{pk}}$.} Reordering $x_1,\ldots,x_n$, we may assume that only the first $m$ elements $x_1,\ldots,x_m$ of the $p$-basis appear on the list. If we denote by $r_i$ the number of times that each $x_i$ appears on the list, then $1\leq r_i<p$ and $r_1+\cdots+r_m=pk+1$. Thus, using the simmetry of Lemma \ref{bracket properties}(1), we can rearrange the terms inside of the brackets as
\begin{align*}
	[x_{i_0},[x_{i_1},[\ldots[x_{i_{pk}},D]\ldots]]]] =[\underbrace{x_1,[\ldots,[x_1}_{r_1},[\underbrace{x_2,[\ldots,[x_2}_{r_2},[\ldots,[\underbrace{x_m,[\ldots,[x_m}_{r_m},D]\ldots]]\ldots]]\ldots]]]\ldots]].
\end{align*}
Since $[x_{i_0},[x_{i_1},[\ldots[x_{i_{pk}},D]\ldots]]]]$ is an $A$-linear operator from $C$ to $M$, it is enough to show that its evaluation at $\lambda \x^\alpha$ is zero, for any $\lambda\in C^{[1]}$ and any reduced monomial $\x^{\alpha}$. 

\medskip

Let $\lambda \x^\alpha\in C$, where $\lambda\in C^{[1]}$ and $\x^{\alpha}$ is a reduced monomial. For any positive integer $r$, we denote by $[r]$ the set $\{1,\dots,r\}$. Then by Lemma \ref{grouped bracket development}, we obtain: 
\begin{align*}
	&[\underbrace{x_1,[\ldots,[x_1}_{r_1},[\underbrace{x_2,[\ldots,[x_2}_{r_2},[\ldots,[\underbrace{x_m,[\ldots,[x_m}_{r_m},D]\ldots]]\ldots]]\ldots]]]\ldots]](\lambda \x^\alpha)\\
	&=\sum_{\substack{S_i\subset [r_i]\\ i=1,\dots,m }}\prod_{i=1}^m (-1)^{r_i-|S_i|} x_i^{|S_i|} \cdot D\left(\lambda\x^\alpha \cdot \prod_{i=1}^m x_i^{r_i-|S_i|} \right)\\
	&=\sum_{\substack{S_i\subset [r_i]\\ i=1,\dots,m }} \prod_{i=1}^m(-1)^{r_i-|S_i|} x_i^{|S_i|} \cdot D\left(\lambda x_{m+1}^{\alpha_{m+1}}\cdots x_n^{\alpha_n} \cdot \prod_{i=1}^m x_i^{r_i+\alpha_i-|S_i|}
	\right)\ .
\end{align*}
By property (*) stated at the beginning of the proof, we see that the term $x_{m+1}^{\alpha_{m+1}}\cdots x_n^{\alpha_n}$ can be factored out from the expression. Thus, $[x_{i_0},[x_{i_1},[\ldots[x_{i_{pk}},D]\ldots]]]](\lambda \x^\alpha)=0$ if the following expression is also zero: 
\begin{align}\label{evaluation}
	 \sum_{\substack{S_i\subset [r_i]\\ i=1,\dots,m }} \prod_{i=1}^m(-1)^{r_i-|S_i|} x_i^{|S_i|} \cdot D\left(\lambda \cdot \prod_{i=1}^m x_i^{r_i+\alpha_i-|S_i|}
	 \right)\ .
\end{align}
We now consider two subcases depending on the values $r_i+\alpha_i$, for $i=1,\dots,m$.

\medskip

\textit{Case II.1: There is an index $i\in\{1,\dots,m\}$ such that $r_i+\alpha_i<p$.} Without loss of generality, we may assume that $i=1$. 
Then, in each term of the expression, (\ref{evaluation})
 we have
\begin{align*}
D\left(\lambda \cdot \prod_{i=1}^m x_i^{r_i+\alpha_i-|S_i|} \right)=x_1^{r_1+\alpha_1-|S_1|} \cdot D\left(\lambda \cdot \prod_{i=2}^m x_i^{r_i+\alpha_i-|S_i|}\right).
\end{align*}
This follows from property (*), since in this case $x_1^{r_1+\alpha_1-|S_1|}$ is a reduced monomial.
Thus (\ref{evaluation}) becomes
\begin{align*}
	& \sum_{\substack{S_i\subset [r_i]\\ i=1,\dots,m }} (-1)^{r_1-|S_1|} x_1^{r_1+|S_1|}\prod_{i=2}^m(-1)^{r_i-|S_i|} x_i^{|S_i|} \cdot D\left(\lambda \cdot \prod_{i=1}^m x_i^{r_i+\alpha_i-|S_i|}
	 \right)\\
&=\left(\sum_{S_1\subset [r_1]} (-1)^{r_1-|S_1|}\right) x_1^{r_1+\alpha_1}\sum_{\substack{S_i\subset [r_i]\\ i=2,\dots,m }} \prod_{i=1}^m(-1)^{r_i-|S_i|} x_i^{|S_i|} \cdot D\left(\lambda \cdot \prod_{i=1}^m x_i^{r_i+\alpha_i-|S_i|}
	 \right).
\end{align*}
This term is zero, because
\begin{align*}
	 \sum_{S_1\subset [r_1]} (-1)^{r_1-|S_1|} &= \sum_{l=0}^{r_1} (-1)^{r_1-l}\binom{r_1}{l} =(1-1)^{r_1} = 0.
\end{align*}

\medskip

\textit{Case II.2: For every $i\in\{1,\dots,m\}$, we have $r_i+\alpha_i\geq p$.} For a collection $S=(S_1,\dots,S_m)$ of subsets $S_i\subset [r_i]$, we define $U_S:=\{1\leq i \leq m : r_i+\alpha_i-|S_i|\geq p\}$. Then we can express (\ref{evaluation}) as follows:
\begin{align*}
\sum_{\substack{S=(S_1,\dots,S_m) \\ S_i\subset [r_i]} } \prod_{i=1}^m(-1)^{r_i-|S_i|} x_i^{|S_i|} \cdot D\left(\lambda \cdot \prod_{i\in U_S} x_i^{r_i+\alpha_i-|S_i|} \prod_{i\notin U_S} x_i^{r_i+\alpha_i-|S_i|} \right)\ .
\end{align*}
Note that $x_i^{r_i+\alpha_i-|S_i|}$ is a reduced monomial for all $i\notin U_S$. Hence, using property (*), we can rewrite the above expression as:
\begin{align}\label{evaluation3}
	\sum_{\substack{S=(S_1,\dots,S_m) \\ S_i\subset [r_i]} } \prod_{i=1}^m(-1)^{r_i-|S_i|} x_i^{|S_i|} \prod_{i\notin U_S} x_i^{r_i+\alpha_i-|S_i|} \cdot D\left(\lambda \cdot \prod_{i\in U_S} x_i^{r_i+\alpha_i-|S_i|} \right)\ .
\end{align}
Observe that 
 if $i\in U_S$, then $r_i+\alpha_i-|S_i|= p+l_i$ for some $l_i<p$, since $r_i,\alpha_i<p$. 
 Thus, we can write 
$$\lambda\cdot\prod_{i\in U_S} x_i^{r_i+\alpha_i-|S_i|}=\lambda\cdot \prod_{i\in U_S} x_i^p \prod_{i\in U_S} x_i^{l_i}, \quad \text{ where $\lambda\prod_{i\in U_S}x_i^p\in C^{[1]}$ \quad and \quad $\prod_{i\in U_S}x_i^{l_i}$ is a reduced monomial.}$$   
Using this factorization in (\ref{evaluation3}) and applying the definition of $D$, we obtain the following equivalent expression:
\begin{align*}
	\sum_{\substack{S=(S_1,\dots,S_m) \\ S_i\subset [r_i]} } \prod_{i=1}^m(-1)^{r_i-|S_i|} x_i^{|S_i|} \prod_{i\notin U_S} x_i^{r_i+\alpha_i-|S_i|} \prod_{i\in U_S}x_i^{l_i} \cdot \partial \left(\lambda \cdot \prod_{i\in U_S}x_i^p  \right)\ .
\end{align*}
Recall that $l_i= r_i+\alpha_i-|S_i|-p$ for $i\in U_S$, so by regrouping some terms, the above expression becomes
\begin{align*}
	&\sum_{\substack{S=(S_1,\dots,S_m) \\ S_i\subset [r_i]} } \prod_{i=1}^m(-1)^{r_i-|S_i|} \prod_{i\notin U_S} x_i^{r_i+\alpha_i} \prod_{i\in U_S}x_i^{r_i+\alpha_i-p} \cdot \partial \left(\lambda \cdot \prod_{i\in U_S}x_i^p  \right)\\
	&= \sum_{\substack{S=(S_1,\dots,S_m) \\ S_i\subset [r_i]} } \prod_{i=1}^m(-1)^{r_i-|S_i|} x_i^{r_i+\alpha_i-p} \prod_{i\notin U_S} x_i^p \cdot \partial \left(\lambda \cdot \prod_{i\in U_S}x_i^p  \right)\\
	&= \left(\prod_{i=1}^m x_i^{r_i+\alpha_i-p}\right) \cdot \sum_{\substack{S=(S_1,\dots,S_m) \\ S_i\subset [r_i]} } \prod_{i=1}^m(-1)^{r_i-|S_i|} \prod_{i\notin U_S} x_i^p \cdot \partial \left(\lambda \cdot \prod_{i\in U_S}x_i^p  \right)\ .
\end{align*}
By rearranging the above sum over $S=(S_1,\dots,S_m)$ according to the resulting subsets $U_S\subset [m]$, the  expression becomes
\begin{align}
	\nonumber &\left(\prod_{i=1}^m x_i^{r_i+\alpha_i-p}\right) \cdot \sum_{U\subset [m]} \sum_{\substack{S=(S_1,\dots,S_m) \\ S_i\subset [r_i], \ U_S=U }} \prod_{i=1}^m(-1)^{r_i-|S_i|} \prod_{i\notin U} x_i^p \cdot \partial \left(\lambda \cdot \prod_{i\in U}x_i^p  \right)\\ 
 \label{equation5}
 = &\left(\prod_{i=1}^m x_i^{r_i+\alpha_i-p}\right)\cdot  \sum_{U\subset [m]}\sigma_U\cdot   \prod_{i\notin U} x_i^p \cdot \partial \left(\lambda \cdot \prod_{i\in U}x_i^p  \right),
\end{align}
where 
\begin{align*}
	\sigma_U:=\sum_{\substack{S=(S_1,\dots,S_m) \\ S_i\subset [r_i], \ U_S=U }} \prod_{i=1}^m(-1)^{r_i-|S_i|}= \prod_{i\in U} \left(\sum_{\substack{S_i\subset [r_i], \\ r_i+\alpha_i-|S_i|\geq p }} (-1)^{r_i-|S_i|}\right) \prod_{i\notin U} \left( \sum_{\substack{S_i\subset [r_i], \\ r_i+\alpha_i-|S_i|< p }}(-1)^{r_i-|S_i|}\right).
\end{align*}
Note that the last equality follows since finite sums can be exchanged with finite products and the subsets $S_1,\dots,S_m$ are independent of each other in the expression of $\sigma_U$.

We now claim that $\sigma_U=(-1)^{|U|}\cdot \mu$, for some integer $\mu$ independent from the subset $U\subset [m]$.  
Suppose that $U\subset[m]$ is a proper subset and let $V=U\cup \{j\}\subset [m]$ for some $j\notin U$.
Observe that $\sigma_U$ and $\sigma_V$ have the same factors except the one corresponding to $j$. Indeed,
$$\sigma_U=\left(\sum_{\substack{S_j\subset [r_j], \\ r_j+\alpha_j-|S_j|< p }} (-1)^{r_j-|S_j|}\right)\prod_{i\in U} \left(\sum_{\substack{S_i\subset [r_i], \\ r_i+\alpha_i-|S_i|\geq p }} (-1)^{r_i-|S_i|}\right) \prod_{\substack{i\notin U \\ i\neq j}} \left( \sum_{\substack{S_i\subset [r_i], \\ r_i+\alpha_i-|S_i|< p }}(-1)^{r_i-|S_i|}\right),$$
$$\sigma_V=\left(\sum_{\substack{S_j\subset [r_j], \\ r_j+\alpha_j-|S_j|\geq p }} (-1)^{r_j-|S_j|}\right)\prod_{i\in U} \left(\sum_{\substack{S_i\subset [r_i], \\ r_i+\alpha_i-|S_i|\geq p }} (-1)^{r_i-|S_i|}\right) \prod_{\substack{i\notin U \\ i\neq j}} \left( \sum_{\substack{S_i\subset [r_i], \\ r_i+\alpha_i-|S_i|< p }}(-1)^{r_i-|S_i|}\right),$$
and the last two products on both expressions coincide. Therefore, we have that $\sigma_U+\sigma_V=0$, because
$$\sum_{\substack{S_j\subset [r_j], \\ r_j+\alpha_j-|S_j|< p }} (-1)^{r_j-|S_j|} + \sum_{\substack{S_j\subset [r_j], \\ r_j+\alpha_j-|S_j|\geq p }} (-1)^{r_j-|S_j|}= \sum_{S_j\subset [r_j]} (-1)^{r_j-|S_j|}=0.$$
This implies that $\sigma_V=-\sigma_U$.
It follows that $\sigma_U=(-1)^{|U|}\cdot\mu$, with $\mu=\sigma_\emptyset$ for all $U$.

\medskip

Thus, the expression in (\ref{equation5}) simplifies to 
\begin{align*}
&\left(\prod_{i=1}^m x_i^{r_i+\alpha_i-p}\right)\cdot
\sum_{U\subset [m]}
(-1)^{|U|}\cdot \mu \cdot 
\prod_{i\notin U} x_i^{p} \cdot\partial\left(\prod_{i\in U} x_i^p\lambda\right). 
\end{align*}
Finally, using the formula from Lemma \ref{bracket development}, we note that the above expression coincides with
\begin{align}\label{final evaluation}
	\mu \cdot \left(\prod_{i=1}^m x_i^{r_i+\alpha_i-p}\right)[x_1^{p},[\ldots,[x_m^p,\partial]\ldots]](\lambda),
\end{align}
which is our final expression for (\ref{evaluation}). To conclude recall that $pk+1=r_1+\cdots+r_m<pm$, so that $m\geq k+1$. Since $\partial \colon C^{[1]}\to M$ is a differential operator of order $\leq k$, the above term in (\ref{final evaluation}) is zero. 

We have covered all the cases, hence the proof is complete.
\end{proof}

\begin{corollary}\label{cor: restriction of diff surjective}
    Let $A\subset C$ be a finite ring extension of finite exponent. If $C$ has a $p$-basis over $C^{[1]}$, then for each $k\geq 0$ the restriction map $\on{res} \colon \Diff_A^{pk}(C,M)\to \Diff_A^k(C^{[1]},M)$ is surjective for any $C$-module $M$.
\end{corollary}

\medskip 

\subsection{Purely inseparable extensions and differential operators}\label{sec: purely inseparable and differential operators}

 The aim of this subsection is  to give a new characterization of purely inseparable extensions in terms of the modules of differential operators (Theorem \ref{thm: characterization of purely inseparable extensions}). Moreover, we re-obtain the characterization for purely inseparable extensions given by P. S. de Salas in \cite[Theorem 3.8]{PSan99} using different techniques. We study first the modules of differential operators and the modules of principal parts of a purely inseparable extension. Then we use the result on the extension of differential operators (Proposition \ref{prop: extension of differential operators}) to complete the proof of Theorem \ref{purely inseparable new characterization} stated in the introduction.

\medskip

We begin by studying the differential operators and the modules of principal parts
for  a certain subclass of purely inseparable extensions. For the next proposition, observe that given a $C$-module $M$, an element $m\in M$ and a differential operator $D\colon C\to C$ of order $\leq k$, the map $E\colon C\to M$ given by $E(c)=D(c)m$ is a differential operator of order $\leq k$. It is the composition of $D$ and right multiplication by $m$. It will be denoted by $m\cdot D$.

\begin{proposition}\label{prop: differential operators of modular split extension}
    Let $A$ be a ring and let $C=A[x_1,\ldots,x_n]$ be an $A$-algebra with presentation $x_1^{p^{e_1}}=\cdots=x_n^{p^n}=0$ for some positive integers $e_1,\dots,e_n$.  Set $\mathcal{A}:=\{\alpha=(\alpha_1,\ldots,\alpha_n)\in\mathbb{N}^n : 0\leq \alpha_j< p^{e_j}, \ j=1,\dots,n\}$. Then
    \begin{enumerate}
        \item  $\on{P}_{C/A}^k$ is a finite free $C$-module for all $k\geq 0$.
        \item For each $\alpha\in\mathcal{A}$, the $A$-linear map $\Delta_{\alpha}\colon C\to C$ defined as
	\begin{align*}
		\Delta_{\alpha}(\x^{\beta})=\binom{\beta}{\alpha}\x^{\beta-\alpha}
	\end{align*}
        is a differential operator of order $|\alpha|$.

        \item Let $M$ be a $C$-module. Then each differential operator $D\colon C\to M$ of order $\leq k$ can be written uniquely as
	\begin{align*}
            D=\sum_{\substack{\alpha\in\mc{A}\\|\alpha|\leq k}} m_\alpha \cdot \Delta_\alpha,\quad m_\alpha\in M.
	\end{align*}
    \end{enumerate}
\end{proposition}
\begin{proof}
    (1): The hypothesis implies that $A$-module $C$ has basis $\{\x^\alpha:\ \alpha\in \mc{A}\}$.
        Since presentations of an algebra are preserved under base change, the $C$-algebra $C\otimes_A C$ obtained from  $C$ by application of $C\otimes_A-$ has presentation $C\otimes_A C=C[1\otimes x_1,\ldots,1\otimes x_n]$ with defining relations $(1\otimes x_i)^{p^{e_i}}=0$, $i=1,\ldots,n$. 
    We can make a change of generators and consider instead $z_1,\ldots,z_n$, where $z_i:=1\otimes x_i-x_i\otimes 1$. The  defining relations are $z_i^{p^{e_i}}=0$,  $i=1,\ldots,n$. 
    It follows that $C\otimes_A C$ is a free $C$-module with basis given by the monomials $\{\mathbf{z}^{\alpha}: \alpha\in\mc{A}\}$.

    Let $J$ be the kernel of the multiplication map $C\otimes_A C\to C$. It is generated as an ideal by $z_1,\ldots,z_n$; hence for any $k$ the ideal
    $J^{k+1}$ is a free $C$-module generated by $\{\mathbf{z}^{\alpha}: \alpha\in\mc{A}, |\alpha|\geq k+1\}.$
    It follows that $\on{P}_{C/A}^k$ is a free $C$-module generated by the classes of $\{\mathbf{z}^{\alpha}: \alpha\in\mc{A}, |\alpha|\leq k\}$.

    \medskip
    
    (2): Let $\{\varphi_\alpha\colon C\otimes_A C\to C: \ \alpha\in\mc{A}\}$ be the dual basis of $\{\z^\alpha :  \alpha\in \mc{A}\}$, and let $\Delta_\alpha\colon C\to C$ be the corresponding $A$-endomorphisms of $C$.
    Observe that for any $\alpha\in \mc{A}$ we have
    \begin{align*}
        \Delta_\alpha(\x^\beta)=\varphi_\alpha(1\otimes \x^\beta)=\varphi_{\alpha}((\z+\x\otimes 1)^\beta)=\varphi_\alpha(\sum_\gamma \binom{\beta}{\gamma} \x^{\beta-\gamma}\z^\gamma)=\binom{\beta}{\alpha} \x^{\beta-\alpha}.
    \end{align*}
    Note also that $\Delta_\alpha$ is differential operator of order $\leq k$ if and only if $\varphi_\alpha$ is zero on $J^{k+1}$ if and only if $|\alpha|\leq k$. Hence $\Delta_\alpha$ is a differential operator of order $|\alpha|$. 
 
    \medskip
    
    (3): Let $M$ be a $C$-module. 
    Any homomorphism of $C$-modules $\varphi \colon C\otimes_A C\to M$ can be written uniquely as
    \begin{align*}
	\varphi=\sum_{\substack{\alpha\in\mc{A}}} m_\alpha \cdot\varphi_\alpha
	\end{align*}
 for certain $m_\alpha\in M$, $\alpha\in \mc{A}$. Namely, $m_\alpha=\varphi(\z^\alpha)$.
 By applying the identification $\on{Hom}_A(C,M)\cong \on{Hom}_C(C\otimes_A C,M)$
 we obtain that any $D\in\on{Hom}_A(C,M)$
 can be written uniquely as  $D=\sum_{\alpha\in \mc{A}}m_\alpha\cdot\Delta_\alpha$ for certain $m_\alpha\in M$. 
 Finally, $D$ is a differential operator of order $\leq k$ if and only if the corresponding map $\varphi:C\otimes_AC\to M$ is zero on $J^{k+1}$ if and only if $m_\alpha=0$ for $|\alpha|\geq k+1$.
%
%
\end{proof}

Using the above proposition we obtain the following conclusions  on the modules of principal parts and the modules of differential operators  for purely inseparable extensions of rings.
\begin{proposition}\label{prop: purely inseparable implies direct summand}
    Let $A\subset C$ be a purely inseparable extension. Then
    \begin{enumerate}
        \item $\on{P}_{C/A}^k$ is a finitely generated projective $C$-module for all $k\geq 0$.
        \item $\Diff_A^k(C)$ is a $C$-module direct summand of $\End_A(C)$ for all $k\geq 0$.
    \end{enumerate}
\end{proposition}

\begin{proof}
We prove (1).
$\on{P}_{C/A}^k$ is a $C$-module of finite presentation by Proposition \ref{prop: principal parts and finite presentation}. Therefore, it is enough to show that for each prime ideal $\p\subset A$, $C_\p\otimes_C\on{P}_{C/A}^k$ is a flat $C_\p$-module, or equivalently, by Proposition \ref{prop: principal parts and base change}, that $\on{P}_{C_\p/A_\p}^k$ is a flat $C_\p$-module.

Since $A\subset C$ is purely inseparable, there is a faithfully flat map $A_{\p}\to A'$ such that 
    \begin{equation}\label{eq: purely inseparable presentation}
        C':=C_{\p}\otimes_{A_{\p}} A'= A'[X_1,\ldots,X_n]/\langle X_1^{p^{e_1}},\ldots,X_n^{p^{e_n}}\rangle
    \end{equation}
    for some sequence of positive integers $e_1,\dots, e_n$. By Proposition \ref{prop: differential operators of modular split extension}, we know that $\on{P}^k_{C'/A'}$ is a finite free $C'$-module. 
    By Proposition \ref{prop: principal parts and base change} we also know that $\on{P}_{C'/A'}^k=C'\otimes_{C_{\p}} \on{P}_{C_{\p}/A_{\p}}^k$. Finally, note that $C_{\p}\to C'$ is faithfully flat since it is a base change of the faithfully flat map $A_{\p}\to A'$.
   Therefore, by faithfully flat descent, $\on{P}_{C_{\p}/A_{\p}}^k$ is a flat $C_{\p}$-module, as we wanted to show. 

   Property (2) follows easily from (1). Indeed, the exact sequence $0\to J^{k+1}\to C\otimes_A C\to \on{P}_{C/A}^k\to 0$, where $J$ is the kernel of the multiplication map $C\otimes_A C\to C$, is split, by (1). Therefore, $\on{Hom}_C(\on{P}_{C/A}^k,C)$ is a $C$-module direct summand of $\Hom_C(C\otimes_A C,C)$, or equivalently, $\Diff_A^k(C)$ is a $C$-module direct summand of $\End_A(C)$.
\end{proof}

We will show that the converse to the later proposition is true for any finite extension of rings $A\subset C$ with finite exponent. First, we will prove that if the modules of principal parts $\on{P}^{p^e}_{C/A}$ are finite projective $C$-modules, then the extension $A\subset C$ is purely inseparable. 
In order to prove it, we discuss some properties of the modules of principal parts. In what follows, for each positive integer $e$ we shall denote by $$\delta_k^{[e]} \colon C^{[e]}\to \on{P}_{C^{[e]}/A}^k$$ the universal differential operator  of order $k$ of the $A$-algebra $C^{[e]}$.
We discuss some relation between $\on{P}_{C/A}^{pk}$ and $\on{P}_{C^{[1]}/A}^k$. Note that the former can be seen as a $C^{[1]}$-module by restricting scalars.

\begin{lemma}\label{lem: principal parts and frobenius}
    Let $C$ be an $A$-algebra, and let $k\geq 0$. Then there is a unique homomorphism of $C^{[1]}$-modules $\on{P}_{C^{[1]}/A}^k\to \on{P}_{C/A}^{pk}$
    making the following diagram
    \begin{equation*}
    \begin{tikzcd}
        C\arrow[r,"\delta_{pk}"]& \on{P}_{C/A}^{pk}\\
        C^{[1]}\arrow[u, hook ] \arrow[r,"\delta_k^{[1]}"]& \on{P}_{C^{[1]}/A}^k\arrow[u]
        \end{tikzcd}
    \end{equation*}
    commutative.
    \end{lemma}
  \begin{proof}
	The restriction of $\delta_{pk}$ to $C^{[1]}$ is a differential operator of order $\leq k$ by Proposition \ref{restriction of differential operator to C1}. Thus, the existence of our map is a consequence of the universal property of $\delta_k^{[1]}$. 
\end{proof}

\begin{proposition}
    \label{prop: retraction for module of principal parts}
    Let $A\subset C$ be a finite ring extension of finite exponent.
    \begin{enumerate}
        \item There is a natural map of $C$-modules
         \begin{align*}
        C\otimes_{C^{[1]}} \on{P}_{C^{[1]}/A}^k\to \on{P}_{C/A}^{pk},\quad 1\otimes \delta_k^{[1]}c^p\mapsto (\delta_{pk})^p.
    \end{align*}
    \item If $C$ has a $p$-basis over $C^{[1]}$, then the above map has a retraction $\on{P}_{C/A}^{pk}\to C\otimes_{C^{[1]}} \on{P}_{C^{[1]}/A}^k$.
    \end{enumerate}
\end{proposition}

\begin{proof}
Part (1) is an immediate consequence of Lemma \ref{lem: principal parts and frobenius}. As for the proof of (2), it is enough to show that for any $C$-module $M$, the induced map
\begin{align*}
    \on{Hom}_C(\on{P}_{C/A}^{pk},M)\to \on{Hom}_C(C\otimes_{C^{[1]}} \on{P}_{C^{[1]}/A}^k,M)=\on{Hom}_{C^{[1]}}(\on{P}_{C^{[1]}/A}^k,M)
\end{align*}
is surjective. By the commutativity of the diagram of Lemma \ref{lem: principal parts and frobenius} and the
universal properties of $\delta_{pk}$ and $\delta_k^{[1]}$, this map can be identified with
the restriction map
\begin{align*}
   \on{res}\colon \Diff_A^{pk}(C,M)\to \Diff_A^k(C^{[1]},M),
\end{align*}
which is surjective by Corollary \ref{cor: restriction of diff surjective}.
\end{proof}
\color{black} 

\begin{proposition}\label{prop: purely inseparable and module of principal parts}
    Let $A\subset C$ be a finite ring extension of finite exponent. If the $C$-module $\on{P}_{C/A}^{p^e}$ is projective for all $0\leq e < \exp(C/A)$, then $A\subset C$ is purely inseparable.
\end{proposition}
\begin{proof}
  We assume now that $\on{P}_{C/A}^{p^e}$ is a projective $C$-module for all $0\leq e < \exp(C/A)$. We prove that $A\subset C$ is purely inseparable by induction on the exponent, $\exp(C/A)$. First of all, note that from the decomposition $\on{P}_{C/A}^1=C\oplus\Omega_{C/A}$ it follows that $\Omega_{C/A}$ is a projective $C$-module. Moreover, since $\Omega_{C/C^{[1]}}=\Omega_{C/A}$ and the extension $C^{[1]}\subset C$ is finite, it follows that $C^{[1]}\subset C$ is Galois (Theorem \ref{thm: Galois characterizations}). In particular, if $\exp(C/A)=1$, this proves that $A\subset C$ is Galois and hence, purely inseparable.
    
    We now assume that $\exp(C/A)>1$. We are going to prove that
    $\on{P}_{C^{[1]}/A}^{p^e}$ is a projective $C^{[1]}$-module for $0\leq e <\exp(C^{[1]}/A)=\exp(C/A)-1$. 
        Since $C^{[1]}\subset C$ is Galois and $\Spec(A)$ and $\Spec(C^{[1]})$ are homeomorphic, it follows from Proposition \ref{prop: Galois iff omega projective} that
        there are $f_1,\ldots,f_r\in A$ such that $C_{f_i}$ has a $p$-basis over $C_{f_i}^{[1]}$. Therefore, for our purpose, we may assume that $C$ has a $p$-basis over $C^{[1]}$.  By Proposition \ref{prop: retraction for module of principal parts}, the natural map
    \begin{align*}
	C\otimes_{C^{[1]}} P_{C^{[1]}/A}^{p^e}\to P_{C/A}^{p^{e+1}}
    \end{align*}
    has a retraction of $C$-modules. This implies that $C\otimes_{C^{[1]}} P_{C^{[1]}/A}^{p^e}$ is a $C$-module direct summand of $P_{C/A}^{p^{e+1}}$, and therefore, it is a projective $C$-module for all $0\leq e< \exp(C^{[1]}/A) $. Since $C^{[1]}\to C$ is faithfully flat, it follows that $P_{C^{[1]}/A}^{p^e}$ is a projective $C^{[1]}$-module for all $0\leq e<\exp(C^{[e]}/A)$. We can now apply the inductive hypothesis to deduce that $A\subset C^{[1]}$ is purely inseparable. Since $C^{[1]}\subset C$ is Galois, we conclude that $A\subset C$ is purely inseparable by Corollary \ref{cor: purely inseparable inductively}.
\end{proof}

Finally, we use the result on extension of differential operators to prove that an $\mc{F}$-extension is purely inseparable if the modules of differential operators are direct summands of the module of endomorphisms. 

\begin{proposition}\label{prop: direct summand implies purely inseparable}
Let $A\subset C$ be an $\mc{F}$-extension such that $\Diff_A^{p^e}(C)$ is a $C$-module direct summand of $\End_A(C)$ for all $0\leq e<\exp(C/A)$. Then $A\subset C$ is purely inseparable.  \end{proposition}

\begin{proof}
We proceed by induction on the exponent $\exp(C/A)$. The exponent one case was already settled in Proposition \ref{prop: direct summand of endomorphism implies Galois}. Indeed, the hypothesis in this case are that $C$ is a projective $A$-module and that $\on{Diff}_A^1=C\oplus \on{Der}_A(C)$ is a $C$-module direct summand of $\on{End}_A(C)$, and the latter implies that $\on{Der}_A(C)$ is a $C$-module direct summand of $\on{End}_A(C)$ too.

Assume now that $\exp(C/A)>1$. Since $A\subset C$ is an $\mc{F}$-extension, $C$ is a finite projective $A$-module. In particular, $C$ has finite presentation over $A$. This implies the following
 \begin{enumerate}
     \item $A\subset C$ is purely inseparable if and only if the extension $A_{\p}\subset C_{\p}$ is purely inseparable for each  prime ideal $\p\subset A$.
     \item $(\End_A(C))_{\p}=\End_{A_{\p}}(C_{\p})$ and $(\Diff_A^{p^e}(C))_{\p}=\Diff_{A_{\p}}^{p^e}(C_{\p})$ (see Corollary \ref{localization of End} and Corollary \ref{cor: localization of diff}).
 \end{enumerate}
 The hypothesis implies that $\Diff_{A_{\p}}^{p^e}(C_{\p})$ is a $C_{\p}$-direct summand of $\End_{A_{\p}}(C_{\p})$, since localization commutes with direct sum.  By Remark \ref{rem: base change and F-extensions}, we also know that $A_{\p}\subset C_{\p}$ is an $\mc{F}$-extension. Therefore, for the rest of the proof we can assume that $A$ is local, and hence $C$ and $C^{[e]}$ are also local for all $e\geq 0$.

 Recall that the extension $A\subset C$ is purely inseparable if and only if $A\subset C^{[1]}$ is purely inseparable and $C^{[1]}\subset C$ is Galois (Corollary \ref{cor: purely inseparable inductively}).
   We prove first that $C^{[1]}\subset C$ is a Galois extension. In fact, $C$ is a projective $C^{[1]}$-module because $A\subset C$ is an $\mc{F}$-extension. In addition, in the following chain of $C$-modules
   \begin{align}
      \on{Diff}_A^1(C)= \on{Diff}_{C^{[1]}}^1(C)\subset \on{End}_{C^{[1]}}(C)\subset \on{End}_A(C)
   \end{align}
    $\on{Diff}_A^1(C)$ is a $C$-module direct summand of $\on{End}_A(C)$ by hypothesis, and $\on{End}_{C^{[1]}}(C)$ is a $C$-module direct summand of $\End_A(C)$ because $C$ is a projective $C^{[1]}$-module (see Proposition \ref{prop: endomorphisms of projective submodule} for details). Therefore, $\on{Diff}_{C^{[1]}}^1(C)$ is a $C$-module direct summand of $\on{End}_{C^{[1]}}(C)$. By the exponent one case, $C^{[1]}\subset C$ is Galois.

We now prove that $A\subset C^{[1]}$ is purely inseparable. Since $\exp(C^{[1]}/A)=\exp(C/A)-1$, it is enough to show that $A\subset C^{[1]}$ satisfies the hypothesis on the statement and apply the inductive hypothesis. It is clear that $A\subset C^{[1]}$ is an $\mc{F}$-extension, since so is $A\subset C$. Thus, it is only left to prove that $\Diff_A^{p^e}(C^{[1]})$ is a $C^{[1]}$-direct summand of $\End_A(C^{[1]})$ for all $0\leq e <\exp(C^{[1]}/A)$.  There is a commutative diagram 
\begin{equation*}
 \begin{tikzcd}[row sep=huge]
\arrow[bend left=30, swap]{d}{\textrm{r}}      \End_A(C)\arrow[r,"\textrm{res}"] &\Hom_A(C^{[1]},C)\arrow[bend right=30, swap]{l}{\textrm{ext}}& \End_A(C^{[1]})\arrow{l}\\
\Diff_A^{p^{e+1}}(C) \arrow[r,"\textrm{res}"] \arrow{u} &  \arrow[bend right=30, swap]{l}{\textrm{ext}} \arrow{u}\Diff_A^{p^e}(C^{[1]},C) \arrow[bend left=30, swap]{r}{\textrm{r}} &\Diff_A^{p^e}(C^{[1]})\arrow{l}\arrow{u},
  \end{tikzcd}
 \end{equation*}
where $\on{res}$ denotes the restriction maps as in Proposition \ref{restriction of differential operator to C1}, $\on{ext}$ is the map defined in Proposition \ref{prop: extension of differential operators}, the unnamed maps  are natural inclusions, and $r$ denotes a retraction to the corresponding inclusion that we now explain. The existence of the vertical retraction on the left of the diagram follows from the hypothesis that $\on{Diff}_A^{p^{e+1}}(C)$ is a $C$-module direct summand of $\on{End}_A(C)$. The existence of the horizontal retraction map can be shown as follows. We have proved that $C$ has a $p$-basis $x_1,\ldots,x_n$ over $C^{[1]}$, so there is a $C^{[1]}$-module map $C\to C^{[1]}$ given by $\sum \lambda_\alpha \mathbf{x}^\alpha\mapsto \lambda_0$. Note that this is a retraction to the inclusion $C^{[1]}\to C$, so it induces a retraction $\Diff_A^{p^e}(C^{[1]},C)\to \Diff_A^{p^e}(C^{[1]})$ to the reverse inclusion.

All the maps in the above diagram are morphisms of $C^{[1]}$-modules.
It is easy to check now that the map from $\on{End}_A(C^{[1]})$ to $\on{Diff}_A^{p^e}(C^{[1]})$ deduced from the above diagram is a retraction of the reverse inclusion. We conclude that  $\Diff_A^{p^e}(C^{[1]})$ is a $C^{[1]}$-direct summand of $\End_A(C^{[1]})$, as we wanted. This completes the proof.
\end{proof}

To sum up, as a consequence of the already presented results, we obtain the next characterizations of purely inseparable extensions of rings. 

\begin{theorem}\label{thm: purely inseparable and differential operators}
    Let $A\subset C$ be a finite ring extension of finite exponent. Then the following conditions are equivalent.
    \begin{enumerate}
    \item $A\subset C$ is purely inseparable.
    \item $\on{P}_{C/A}^k$ is a projective $C$-module for all $k$.
    \item $\on{P}_{C/A}^{p^e}$ is projective $C$-module for all $0\leq e < \exp(C/A)$.
    \item  $A\subset C$ is an $\mc{F}$-extension and $\Diff_A^{k}(C)$ is a $C$-direct summand of $\End_A(C)$ for all $k$.
        \item $A\subset C$ is an $\mc{F}$-extension and $\Diff_A^{p^e}(C)$ is a $C$-direct summand of $\End_A(C)$ for all $0\leq e < \exp(C/A)$.
    \end{enumerate}
\end{theorem}

\begin{proof}
By Proposition \ref{prop: purely inseparable implies direct summand}, (1) implies (2) and (4), and clearly  (2) implies (3).
The implication, $(3)\Rightarrow (1)$ 
follows from Proposition \ref{prop: purely inseparable and module of principal parts}. So $(1)\Leftrightarrow(2)\Leftrightarrow(3)$. 
On the other hand, we have that (4) implies (5) and the implication $(5)\Rightarrow(1)$ is proved in Proposition \ref{prop: direct summand implies purely inseparable}.
\end{proof}

The above theorem includes Theorem \ref{purely inseparable new characterization} from the introduction. It also includes the characterization of purely inseparable extensions given by P.S. de Salas in \cite[Theorem 3.8]{PSan99}. However, we obtain it using a completely different method.

\medskip

\subsection{A summary of equivalences of purely inseparable extensions of rings} To conclude, we collect in the following theorem the characterizations of purely inseparable extensions of rings presented in this section. As we announced, this result extends in some sense Theorem \ref{thm: Galois characterizations} for exponent one extensions to the setting of higher exponent extensions. It includes the characterizations from Theorem \ref{thm: characterization of purely inseparable} and those related to differential operators and the modules of principal parts (Theorem \ref{thm: purely inseparable and differential operators}) that we have proved along the section. \

\begin{theorem}\label{thm: characterization of purely inseparable extensions}
    Let $A\subset C$ be a finite ring extension of finite exponent. Then the following are equivalent:
    \begin{enumerate}
        \item $A\subset C$ is a purely inseparable extension.
                \item $C^{[e+1]}\subset C^{[e]}$ is a Galois extension for all $0\leq e <\exp(C/A)$. 
        \item $\Omega_{C^{[e]}/A}$ is a projective $C^{[e]}$-module for all  $0\leq e <\exp(C/A)$. 
        \item $\on{P}_{C/A}^{p^e}$ is projective $C$-module for all $0\leq e < \exp(C/A)$.
        \item $A\subset C$ is an $\mc{F}$-extension and $\Diff_A^{p^e}(C)$ is a $C$-direct summand of $\End_A(C)$ for all $0\leq e < \exp(C/A)$.
    \end{enumerate}
\end{theorem}
\color{black}

\section{The Jacobson-Bourbaki Theorem for ring extensions}\label{sec: Jacobson-Bourbaki}

Given a finite field extension $K\subset L$, the Jacobson-Bourbaki correspondence (\cite[I, \S 2, Theorem 2]{Jacobson64}) establishes a one-to-one correspondence between the intermediate subfields $K\subset E\subset L$ and the unital $K$-subalgebras of $\End_K(L)$ that are also left vector $L$-spaces.  This correspondence, which was originally stated for finite purely inseparable field extensions of exponent one, was quickly extended to the non-commutative case of division rings (\cite{Jacobson47}, \cite{Cartan47}). Later on, a Jacobson-Bourbaki correspondence was formulated for the so called Galois intermediate subfields of a commutative ring (\cite{Winter2005}). Moreover, this formulation has been expanded recently to the setting of non-commutative rings which are finitely generated over their centers (\cite{FengSheng}).  

The general Jacobson-Bourbaki correspondence applies in particular for purely inseparable extensions of fields. In this case, any intermediate field $E$ of a purely inseparable extension $K\subset L$ induces two purely inseparable extensions $K\subset E$ and $E\subset L$. Hence, it is natural to ask if a generalization of this correspondence to the setting of purely inseparable extensions of rings holds. That is, one may ask if there is a one-to-one correspondence between the subrings $B$ of a purely inseparable extension $A\subset C$ such that both $A\subset B$ and $B\subset C$ are purely inseparable and some subojects of $\End_A(C)$. As we will see in Section \ref{section on purely inseparable towers}, finding those subrings seems not to be a simple problem, and we have not achieve yet such a correspondence. However, while dealing with this question, we have been able to formulate and prove a Jacobson-Bourbaki correspondence for a more general class of commutative ring extensions (Theorem \ref{th: endomorphism correspondence}).

More precisely, we  establish a Jacobson-Bourbaki correspondence for finite ring extensions $A\subset C$ such that $C$ is projective as $A$-module and $\Spec(A)$ and $\Spec(C)$ are homeomorphic. Observe that any finite extension of finite exponent has homeomorphic spectra, so, in particular, any purely inseparable extension, belongs to the described class of ring extensions.  We prove that there is a one-to-one correspondence between the intermediate subrings $A\subset B\subset C$ such that $C$ is projective over $B$ and the unital $A$-subalgebras of $\End_A(C)$ that are also $C$-module direct summands. 

\medskip

First, we introduce some preliminaries on the endomorphisms of a projective module that we use later. 
\begin{parrafo}\label{endomorphisms of projective module}
    \textit{Endomorphisms of a projective module.} Let $A\subset C$ be a finite extension such that $C$ is a projective $A$-module. Recall that $\End_A(C)$ is a projective left $C$-module in a natural way (see Corollary \ref{cor: End is finite and projective}) and a unital $A$-algebra under composition.
    Moreover, under this hypothesis, note that the ring $A$ can be recovered as
    \begin{align*}
     A=\{x\in C: [\d,x]=0,\ \forall \d\in \on{End}_A(C)\}.
    \end{align*}
    In fact, it is clear that $[\d,x]=0$ for any $x\in A$, since $\d\in\End_A(C)$ is an $A$-linear map. Conversely, since $C$ is finite and projective over the subring $A$, we have a decomposition $C=A\oplus C/A$ and we have a projection map $\varphi\colon C\to A$ whose restriction $\varphi|_A$ is the identity map (Proposition \ref{projective finite extensions are split}).   Let $x\in C$ such that $[\d,x]=0$ for all $\d\in\End_{A}(C)$, then we have $\d(x)=x\cdot\d(1)$. In particular, as $\varphi$ is an $A$-endomorphism of $C$, we have $x=x\cdot \varphi(1)=\varphi(x)\in A$, as we wanted. \color{black}
\end{parrafo}

Now, we study the sets of endomorphisms with its structure of subalgebras and submodules that are in correspondence to certain intermediate subrings of a projective ring extension. 

\begin{proposition}\label{prop: endomorphisms of projective submodule}
    Let $A\subset C$ be a finite ring extension such that $C$ is projective as $A$-module. Let $A\subset B\subset C$ be an intermediate ring such that $C$ is projective  as $B$-module. Then
    \begin{enumerate}
        \item $\End_B(C)$ is an unital $A$-subalgebra of $\End_A(C)$. 
        \item $\End_B(C)$ is direct summand of $\End_A(C)$ as $C$-module.
    \end{enumerate}
\end{proposition}
\begin{proof}
    (1): Since $A\subset B$, any $B$-endomorphism of $C$ is in particular an $A$-endomorphism. Thus it is clear that $\End_B(C)$ is an $A$-subalgebra of $\End_A(C)$ that contains the same identity element.

    \medskip

     (2): Since $C$ is finite and projective over $B$, $C\otimes_B C$ is finite and projective over $C$. Thus, the surjective map of $C$-modules $C\otimes_A C\to C\otimes_B C$ has a section. Hence, the injective map $\Hom_C(C\otimes_B C,C)\to \Hom_C(C\otimes_A C, C)$, that results from applying $\Hom_C(-,C)$ to the previous map, has a retraction. This implies that $\Hom_C(C\otimes_B C,C)$ is a direct summand of $\Hom_C(C\otimes_A C, C)$. Finally, we note that $\Hom_C(C\otimes_A C, C)\simeq\End_A(C)$ and similarly $\Hom_C(C\otimes_B C, C)\simeq\End_B(C)$. Hence, we have proved that $\End_B(C)$ is a direct summand of $\End_A(C)$ as $C$-module. 
  \end{proof}

    The proof of the next proposition is similar to the one given by Yuan in \cite[Lemma 3.2]{Yuan70/2}, that is also based on Hochschild's proof of the Jacobson-Bourbaki correspondence for division rings (\cite[Theorem 2.1]{Hochschild49}). We make some modifications on this argument to consider endomorphisms and finite projective extensions with homeomorphic spectra.
    
\begin{proposition}\label{prop: direct summands subalgebras of endomorphisms}
    Let $A\subset C$ be a finite ring extension such that $C$ is projective as $A$-module. Assume in addition that the induced morphism $\Spec(C)\to \Spec(A)$ is a homeomorphism.  Let $H\subset \End_A(C)$ be a unital $A$-subalgebra that is also a $C$-module direct summand of the left $C$-module $\End_A(C)$. Put $B_H:=\{x\in C: [\varphi,x]=0,\ \forall \varphi\in H\}$. Then $A\subset B_H\subset C$ is an intermediate ring and $C$ is a projective $B_H$-module. In addition, $H=\End_{B_H}(C)$.
\end{proposition}

\begin{proof}
    Observe that $B_H$ is the subset of elements of $C$ that are constants for every $\varphi\in H$, i.e., $x\in B_H$ if and only if $\varphi(xc)=x\cdot\varphi(c)$ for all $\varphi\in H$ and $c\in C$. Then it is clear that $B_H$ is a ring such that $A\subset B_H\subset C$ and, therefore $\Spec(B_H)$ is also homeomorphic to $\Spec(A)$ and $\Spec(C)$. In order to prove that $C$ is finite and projective as $B_H$-module, we will show that it is locally free of finite rank over $B_H$. 
    
    \medskip
    First, we show that $B_H$ behaves well under localization, i.e., that $S^{-1}B_H=B_{S^{-1}H}$ for any multiplicatively closed subset $S \subset A$. We consider the $A$-module $M:=\Hom_A(H,\End_A(C))$ and define a map $\Phi\colon C\to M$ as $\Phi(x)=[x,-]$. This is a map of $A$-modules, and by definition, $B_H$ is its kernel. Since $H$ is $C$-module direct summand of $\on{End}_A(C)$ and the latter is a finite projective $C$-module, $H$ is also a finite projective $C$-module, and since $C$ is a finite projective $A$-module, $H$ is also a finite projective $A$-module. In particular, $H$ is an $A$-module of finite presentation. It follows that if $S\subset A$ is a multiplicative subset, then $S^{-1}\on{Hom}_A(H,\on{End}_A(C))=\on{Hom}_{S^{-1}A}(S^{-1}H, \on{End}_{S^{-1}A}S^{-1}C)$. Hence $B_{S^{-1}H}=S^{-1}B_H$, as we wanted to show. 

\medskip 

    Now, let $\p\subset A$ be a prime ideal. By Proposition \ref{existence of special basis for a direct summand of End}, there exists $f\in A\setminus\p$, elements $t_1,\dots,t_n\in C$ that induce a basis for $C_f$ as a free $A_f$-module and a basis $\phi_1,\ldots,\phi_l$ for $H_f\subset \on{End}_{A_f}(C_f)$ as $C_f$-module, such that $\phi_i(t_j)=\delta_{ij}$ in $C_f$ for all $1\leq i,j\leq l$. 
 We show that for all $i=1,\ldots,l$ it holds $\on{Im}\phi_i\subset B_{H_f}$, or equivalently, $\phi(\phi_i(c)x)=\phi_i(c)\phi(x)$ for all $c,x\in C_f$ and $\phi\in H_f$.
  To prove this we use that any $\psi\in H_f$ can be expressed as $\psi=\sum_{k=1}^l\psi(t_k)\cdot \phi_k$. 
  Recall also that $H_f$ is a $C_f$-module and $H_f$ is an $A_f$-algebra with the composition. Hence for any $\phi\in H_f$ we have $\phi\circ (x\cdot \phi_i)\in H_f$ for any $x\in C_f$, and so
    $$\phi\circ (x\cdot \phi_i)=\sum_{k=1}^{l}(\phi\circ (x\cdot \phi_i))(t_k)\phi_k=\sum_{k=1}^{l}\phi(x\cdot \phi_i(t_k))\cdot \phi_k=\phi(x)\cdot \phi_i.$$
    In particular, evaluating the above expression on $c\in C_f$, we obtain $\phi(x\phi_i(c))=\phi(x)\phi_i(c)$, which proves the claim.

    We prove now that $t_1,\ldots,t_l$ form a basis for $C_f$ over $B_{H_f}$. Note first that these elements are linearly independent over $B_{H_f}$, because if $b_1t_1+\cdots+b_lt_l=0$ for $b_i\in B_{H_f}$, after applying $\phi_i$ for $i=1,\dots,l$, we get $b_1=\cdots=b_l=0$.
    It remains to show that $t_1,\dots,t_l$ generate $C_f$ as $B_{H_f}$-module. Let $c\in C_f$ and define $c':=c-\sum_{j=1}^l \phi_j(c)t_j$.
    Using that $\phi_j(c)\in B_{H_f}$, we have $\phi_i(c')=\phi_i(c)-\phi_i(c)=0$ for all $i=1,\ldots,l$. Since $H_f$ is a unital $A_f$-subalgebra, the identity endomorphism $\on{id}_{C_f}$ is contained in $H_f$ and it can be expressed as a $C_f$-linear combination of $\phi_1,\dots,\phi_l$. Thus, we have $0=\on{id}_{C_f}(c')=c'$ and $c=\sum_{i=1}^l \phi_i(c)t_i$. 

    We now prove that $H_f=\on{End}_{B_{H_f}}(C_f)$. Any $\phi\in H_f$ is $B_{H_f}$-linear by definition of $B_{H_f}$. Conversely, if $\phi\in\End_{B_{H_f}}(C_f)$ then
    $\phi=\sum_{i=1}^l \phi(t_i)\cdot\phi_i$ since $t_1,\ldots,t_l$ is a basis for $C_f$ over $B_{H_f}$ and $\phi_1,\ldots,\phi_l$ is the dual basis. Hence $\phi\in H$.

\medskip 

    We have shown that locally in the Zariski topology of $\Spec(A)$, $C$ is a is a  free $B_H$-module of finite rank and $H=\on{End}_{B_H}(C)$.
    Hence, $C$ is a projective $B_H$-module and $H=\on{End}_{B_H}(C)$ holds globally. This completes the proof.
\end{proof}

Combining Propositions \ref{prop: endomorphisms of projective submodule} and \ref{prop: direct summands subalgebras of endomorphisms}, we obtain inmediately the following correspondence that generalizes Jacobson-Bourbaki correspondence for field extensions to a more general context of ring extensions:

\begin{theorem}\label{th: endomorphism correspondence}
    Let $A\subset C$ be a finite ring extension such that $C$ is projective as $A$-module. Assume in addition that  the induced morphism $\Spec(C)\to \Spec(A)$ is an homeomorphism. Then there is a bijection between 
	$$\mc{G}_{C/A}:=\{B : \text{$B$ is an intermediate ring $A\subset B\subset C$} \quad \text{and} \quad  \text{$C$ is a projective $B$-module}\}$$
    and 
	$$\mc{D}_{C/A}:=\{H : \text{$H$ is a unital $A$-subalgebra of $\End_A(C)$ and a $C$-module direct summand of $\End_A(C)$} \},$$
	 given by the following inverse maps
	 \begin{center}
	 	\begin{tabular}{c l c l}
	 		$\mc{G}_{C/A}$ & $ \to \mc{D}_{C/A},$ & $\mc{D}_{C/A}$ & $\to \mc{G}_{C/A}$ \\
	 		$B$ & $\mapsto \End_B(C)$ & $H$ & $\mapsto B_H\,,$
	 	\end{tabular}
	 \end{center}
    where $B_H:=\{x\in C: [\varphi,x]=0,\ \forall \varphi\in H\}$. 
\end{theorem}

The following example shows that the theorem does not hold if $\Spec(A)$ and $\Spec(C)$ are not homeomorphic.

\begin{example}\label{ex: homeomorphism is necessary}
    Let $K$ be a field and let $L=K\times K $ with the component-wise multiplication. 
    We can consider $L$ as a ring extension of $K$ via the morphism $x\mapsto(x,x)$. 
    Note that $\Spec(L)$ is not homeomorphic to $\Spec(K)$ since the former contains two points and the latter only one. We will show that the correspondence in Theorem \ref{th: endomorphism correspondence} does not hold in this case.
    
    By making use of the canonical basis of $L=K\times K$, we can think of the elements of $\End_K(L)$ as $2\times 2$ matrices with coefficients in $K$. 
    The elements of  $L$ viewed as endomorphisms are the diagonal matrices. 
    Let
    \begin{align*}
        H=\left\{\begin{pmatrix} a&b\\ 0&d \end{pmatrix}: a,b,c\in K
        \right\}, \quad P=\left\{\begin{pmatrix} 0&0\\ c&0 \end{pmatrix}: c\in K
        \right\}.
    \end{align*}
    It is clear that $H$ and $P$ are $L$-submodules of $\on{End}_K(L)$ and  $\on{End}_K(L)=H\oplus P$.
    It is also clear that $H$ is an unital $K$-subalgebras of $\End_K(L)$. 
   Finally, since the only diagonal matrices that commute with every matrix in $H$ are the scalar multiples of the identity, we have
   $B_{H}=K$. However, $\on{End}_{B_H}(L)=\on{End}_K(L)\neq H$.
   \end{example}

\section{Towers of purely inseparable extensions}\label{section on purely inseparable towers}

If $K\subset L$ is a purely inseparable field extension, then for any intermediate field $K\subset E\subset L$ the extensions $K\subset E$ and $E\subset L$ are also purely inseparable. Conversely, if $K\subset E$ and $E\subset L$ are purely inseparable extensions of fields, then  the extension $K\subset L$ is also purely inseparable. 
On the contrary, these statements are no longer true when we move to the general context of purely inseparable ring extensions.
We introduce the notion of purely inseparable tower to formalize these problems.

\begin{definition}
    Let $A\subset B\subset C$ be a tower of rings. We say that $A\subset B\subset C$ is a \textit{purely inseparable tower} of rings if $A\subset C$, $A\subset B$ and $B\subset C$ are purely inseparable extensions. If $A\subset C$ has exponent one, a purely inseparable tower of rings is also called a \textit{Galois tower}.
\end{definition}

In this section, we address two questions regarding purely inseparable towers of rings.
 In \S \ref{subsection: intermediate rings of a purely inseparable extension}, given a purely inseparable extension of rings $A\subset C$, we consider the problem of characterizing those intermediate rings $B$ such that $A\subset B$ and $B\subset C$ are also purely inseparable, that is, those subrings such that $A\subset B\subset C$ is a purely inseparable tower of rings. 
 When $A\subset C$ has exponent one, this problem is well understood: an intermediate ring $B$ defines a Galois tower $A\subset B\subset C$ if and only if $C$ is a projective $B$-module
 (\cite[Théorème 71]{Andre91}). If $A\subset C$ has exponent greater than one the problem remains open.
 We shall study some sufficient conditions by making use of the notion of $\mathcal{F}$-extension. Namely, in Proposition \ref{prop: on purely inseparable towers with the strong condition of F-extensions}, which is Theorem \ref{some conditions for a purely inseparable tower} from the introduction, we show that if $B$ is a subring of a purely inseparable extension $A\subset C$ such that $A\subset B$, $B\subset C$ and $A\subset B[C^{p^k}]$ are $\mc{F}$-extensions for $1\leq k<\exp(C/B)$, then $A\subset B\subset C$ is a purely inseparable tower. We shall also show that in exponent two the condition on the extensions $A\subset B[C^{p^k}]$ can be omitted. 
In \S \ref{subsection: composition of purely inseparable extension}, we study the question of when the composition of two purely inseparable extensions of rings $A\subset B$ and $B\subset C$ gives a purely inseparable tower $A\subset B\subset C$.  
In Corollary \ref{cor: composition of purely inseparable}, which is Theorem \ref{composition of purely inseparable extensions} from the introduction,  we prove that if $A\subset B$ and $B\subset C$ are purely inseparable and $A\subset B[C^{p^k}]$ is an $\mc{F}$-extension for all $1\leq k\leq \exp(C/B)$, then $A\subset C$ is also purely inseparable.
\color{black}

\medskip

\subsection{Intermediate rings of a purely inseparable extension of rings}\label{subsection: intermediate rings of a purely inseparable extension} 
Given a purely inseparable extension of rings $A\subset C$, we aim to find some conditions for an $A$-subalgebra $B$ to induce a purely inseparable tower $A\subset B\subset C$. The situation in exponent one is well-understood.
 \begin{proposition}[{\cite[Théorème 71]{Andre91}, \cite[Theorem 11]{Yuan70/1}}]\label{exponent one purely inseparable towers}
	Let $A\subset C$ be a Galois extension of exponent one and let $B$ be an intermediate ring. 
 Then $A\subset B\subset C$ is a Galois tower if and only if $C$ is a projective $B$-module.
\end{proposition}

On the other hand, one can easily produce examples $A\subset B\subset C$ where $A\subset C$ and $A\subset B$ are Galois but $B\subset C$ is not. 

\begin{example}
	Let $A=k$ be a field of characteristic $p=2$. We consider the ring $$C:=k[X,Y]/\langle X^2, Y^2\rangle =k[x,y],$$
	which is a purely inseparable extension of exponent one over $A$, with $p$-basis $\{x,y\}$, so $k\subset C$ is a Galois extension. 
 We define the subring $B:=k[xy]\subset k[x,y]=C$. Note that $B$ has rank 2 over $k$ and $xy$ is a $2$-basis, so $k\subset B$ is a Galois extension. 
 We show that $B\subset C$ is not Galois by showing that $C$ is not a flat $B$-module.
  Assume on the contrary that $C$ is a flat module. Since $B$ is a local ring and since $C/\langle xy\rangle=k[X,Y]/\langle X^2, XY, Y^2\rangle$, which has a basis over $B/\langle xy\rangle=k$ induced by $1,x,y$, by Nakayama's lemma $1,x,y$ would be a basis for $C$ as $B$-module. But this is impossible since $(xy)x=0$. Hence $C$ is not a flat $B$-module and therefore it is not a Galois extension.
\end{example}

We focus now on purely inseparable extensions of higher exponent. 
We observe that Proposition \ref{exponent one purely inseparable towers}
does not hold for purely inseparable extensions of exponent greater than one. In fact, in the next example we construct a tower of rings $A\subset B\subset C$ where $A\subset C$ is a purely inseparable extension of exponent 2, $B\subset C$ is a Galois extension, but $A\subset B$ is not purely inseparable.

 \begin{example}\label{example failed purely inseparable tower}
 Let $A=k$ be a field of characteristic $p=3$. We set $$B:=k[X,Y,Z_1,Z_2]/\langle X^3, Y^3, Z_1^3- X^2,Z_2^3-Y^2\rangle=k[x,y,z_1,z_2],$$
 	which is an extension of $A$ of exponent 2. 
  We set 
 	$$C:=B[T_1,T_2]/\langle T_1^3 -x, T_2^3 -y \rangle =k[x,y,z_1,z_2,t_1,t_2].$$
 	By definition $C$ has $p$-basis over $B$, so $B\subset C$ is a Galois extension. Note also that $A\subset C$ is an extension of exponent 2.
  We claim that $A\subset C$ is purely inseparable. By Theorem \ref{thm: characterization of purely inseparable} it is enough to show that $A\subset A[C^3]$ and $A[C^3]\subset C$ are Galois. Note that $A[C^3]=k[x,y]$ is a Galois extension over $A$ with $p$-basis $\{x,y\}$, and 
 	$$C=A[C^3][Z_1,Z_2,T_1,T_2]/\langle Z_1^3-x^2,Z_2^3-y^2, T_1^3-x,T_2^3-y \rangle$$
   is Galois over $A[C^3]$ with $p$-basis $\{z_1,z_2,t_1,t_2\}$. This proves our claim. 
   We now show that $A\subset B$ is not purely inseparable. It is enough to show that $A\subset A[B^3]$ is not Galois, and since $A$ is local it is enough to show that $A[B^p]$ has no $p$-basis over $A$. 
   This is true because
 	$A[B^3]=k[x^2,y^2]$ has a basis over $k$ formed by the four monomials $1,x^2,y^2, (xy)^2$, so the rank over $A$ is not a power of $p=3$.
 \end{example}


Proposition \ref{exponent one purely inseparable towers} has the following extension to higher exponent, which will be used to prove Theorem \ref{some conditions for a purely inseparable tower} in Proposition \ref{prop: on purely inseparable towers with the strong condition of F-extensions}.

\begin{lemma}\label{lem: on purely inseparable towers with B<C of exponent one}
Let $A\subset C$ be a purely inseparable extension and let $B$ be an intermediate ring such that $A\subset B$ and $B\subset C$ are $\mc{F}$-extensions. If $B\subset C$ has exponent one, then $A\subset B$ is purely inseparable and $B\subset C$ is Galois.    
\end{lemma}

\begin{proof}
    We proceed by induction on the exponent of $A\subset C$. If $\exp(C/A)=1$, the lemma reduces to Proposition \ref{exponent one purely inseparable towers}. 
    Assume now that $e:=\exp(C/A)>1$. Since $\exp(C/B)=1$, we have $C^p\subset B$. Let us consider the tower $A[C^p]\subset B\subset C$. Since $A\subset C$ is purely inseparable, the extension $A[C^p]\subset C$ is Galois by Corollary \ref{cor: purely inseparable inductively}, and since by hypothesis $C$ is a projective $B$-module, both $A[C^p]\subset B$ and $B\subset C$ are Galois by
    Proposition \ref{exponent one purely inseparable towers}. 
    
    We now show that the tower of rings $A\subset A[B^p]\subset A[C^p]$ satisfies the hypothesis of the lemma. The extension $A\subset A[C^p]$ is purely inseparable of exponent $e-1$, by Corollary \ref{cor: purely inseparable inductively}. The extension $A\subset A[B^p]$ is an $\mc{F}$-extension by hypothesis because so is $A\subset B$. We finally show that $A[C^p]$ is a projective $A[B^p]$-module.
    In fact, $B$ is a projective $A[B^p]$-module and $C$ is a projective $B$-module by hypothesis, so $C$ is a projective $A[B^p]$-module by Proposition \ref{restriction of scalar of projective modules}. 
    In addition, $C$ is a projective $A[C^p]$-module because $A[C^p]\subset C$ is Galois. Hence $A[C^p]$ is a projective $A[B^p]$-module by Proposition \ref{tower of projective extensions}.
    
    We now apply the inductive hypothesis and conclude that $A\subset A[B^p]$ is purely inseparable and $A[B^p]\subset A[C^p]$ is Galois. Since we also proved that $A[C^p]\subset B$ is Galois in the first paragraph, we conclude that $A[B^p]\subset B$ is Galois by Lemma \ref{lem: basics on Galois extensions}. Hence $A\subset B$ is a purely inseparable extension by Corollary \ref{cor: purely inseparable inductively}. This and the conclusion of the first paragraph complete the proof of the lemma.
\end{proof}

\begin{proposition}\label{prop: on purely inseparable towers with the strong condition of F-extensions}
    Let $A\subset C$ be a purely inseparable extension and let $B$ be an intermediate ring such that $A\subset B$ and $B\subset C$ are both $\mc{F}$-extensions.
    If $A\subset B[C^{p^k}]$ is an $\mc{F}$-extension for all $0\leq k< \exp(C/B)$, then $A\subset B\subset C$ is a purely inseparable tower. Moreover, $A\subset B[C^{p^k}]\subset C$ is a purely inseparable tower for all $0\leq k< \exp(C/B)$.
\end{proposition}

\begin{proof}
    We prove that $A\subset B[C^{p^k}]$ is purely inseparable by induction on $k$. The case $k=0$ is trivial, so let $k>0$ and assume that $A\subset B[C^{p^{k-1}}]$ is purely inseparable. By hypothesis $A\subset B[C^{p^k}]$ is an $\mc{F}$-extension and $B[C^{p^{k-1}}]$ is a projective $B[C^{p^k}]$-module. We can then apply Lemma \ref{lem: on purely inseparable towers with B<C of exponent one} and conclude that $A\subset B[C^{p^k}]$ is purely inseparable. This completes the induction.
    \end{proof}
\begin{remark}\label{rem: on the condition of F-extension}
    If $A\subset B\subset C$ is a purely inseparable tower and $A\subset B$ has exponent one, then $A\subset B[C^{p^k}]$ is an $\mc{F}$-extension for all $k$. This follows from Proposition \ref{prop: composition of a Galois and a purely inseparable} below. We do not know if the same is true when $A\subset B$ has exponent $>1$.
\end{remark}
The following proposition will be used to show that in an exponent two purely inseparable extension $A\subset B\subset C$ are precisely those for which $A\subset B$ and $B\subset C$ are $\mc{F}$-extensions.
\begin{proposition}
    Let $A\subset C$ be a purely inseparable extension of exponent $\geq 2$ and let $B$ be an intermediate ring such that $A\subset B$ and $B\subset C$ are both $\mc{F}$-extensions. Suppose that $A\subset B[C^{p^k}]$ is an $\mc{F}$-extension for all $k\leq \exp(C/A)-2$.  
    Then $A\subset B\subset C$ is a purely inseparable tower.
\end{proposition}
\begin{proof}
    Let $e=\exp(C/A)$. 
    The result will from Proposition \ref{prop: on purely inseparable towers with the strong condition of F-extensions} if we show that $A\subset B[C^{p^{e-1}}]$ is also an $\mc{F}$-extension. We prove this fact. On the one hand, $B$ is projective over $A[(B[C^{p^{e-1}}])^{p^k}]=A[B^{p^k}]$ for all $k>0$. On the other hand, $B[C^{p^{e-1}}]$ is projective over $B$. Hence  $B[C^{p^{e-1}}]$ is projective over  $A[(B[C^{p^{e-1}}])^{p^k}]$ for all $k$.
\end{proof}

As a direct consequence of the above proposition, we obtain Corollary \ref{purely inseparable towers of exponent two} from the introduction.

\begin{corollary}\label{cor: purely inseparable towers of exponent two}
    Let $A\subset C$ be a purely inseparable extension of exponent two and let $B$ be an intermediate ring such that $A\subset B$ and $B\subset C$ are $\mc{F}$-extensions. Then $A\subset B\subset C$ is a purely inseparable tower.
\end{corollary}

\medskip

\subsection{Composition of purely inseparable extensions of rings}\label{subsection: composition of purely inseparable extension} Given a tower of rings extensions $A\subset B\subset C$ with $A\subset B$ and $B\subset C$ purely inseparable, the extension $A\subset C$ is finite of finite exponent.
We consider here the the problem of whether $A\subset C$ is also purely inseparable, that is, whether $A\subset B\subset C$ is a purely inseparable tower. 
As we will see, in general we need to add some extra conditions to reach this conclusion.  

\medskip

To illustrate the problem we analyze the exponent one case. Let $A\subset B$ and $B\subset C$ be Galois extensions of exponent one.
Then $A\subset C$ has exponent one or two. If $A\subset C$ has exponent one, then it is Galois by Lemma \ref{lem: basics on Galois extensions}. However, if $A\subset C$ has exponent two, we cannot assure in general that $A\subset C$ is purely inseparable, as the following example shows.

\begin{example}\label{ex: not p.i. comp of Galois extensions}
	Let $A=k$  be a field of characteristic $p=3$. We define the ring $$B:=k[X,Y]/\langle X^3, Y^3 \rangle = k[x,y].$$ It is clear that $A\subset B$ is a Galois extension with $p$-basis $\{x,y\}$. Now we define the ring
	$$C:=B[Z_1,Z_2,Z_3]/\langle Z_1^3 - x^2, Z_2^3-xy, Z_3^3 - y^2\rangle =k[x,y,z_1,z_2,z_3].$$
	Then the extension $B\subset C$ is Galois with $p$-basis $\{z_1,z_2,z_3\}$. Note that $A[C^3]=k[x^2,xy,y^2]\subset k[x,y]=B$. One can check that $A[C^3]$ is a free $A$-module with a basis formed by the monomials $1,x^2,xy,y^2,(xy)^2$. Thus $A\subset A[C^3]$ has exponent one and $A[C^3]$ has rank $5$ over $A$. Since we are in the local case, if it were a Galois extension, then $A[C^3]$ would be free of rank a power of $3$ (see Definition \ref{def: p-basis}).  Since $5$ is not a power of $3$, $A\subset A[C^3]$ is not Galois. Therefore, by Theorem \ref{thm: characterization of purely inseparable}, $A\subset C$ is not purely inseparable. 	
\end{example}

The next lemma characterizes when a ring extension, that results from the composition of two Galois extensions, is purely inseparable. 

\begin{lemma}\label{lem: composition of Galois extensions}
	Let $A\subset B$ and $B\subset C$ be Galois extensions of rings. Then $A\subset C$ is purely inseparable if and only if $A\subset C$ is an $\mc{F}$-extension. Moreover, in this case, $A[C^p]\subset B$ is also a Galois extension.
\end{lemma}

\begin{proof}
If $A\subset C$ is purely inseparable, then it is in particular an $\mc{F}$-extension
and $A[C^p]\subset C$ is Galois. In addition, there is a tower
$A[C^p]\subset B\subset C$ and $B\subset C$ is Galois by hypothesis. Hence $A[C^p]\subset B$ is also Galois by Proposition \ref{exponent one purely inseparable towers}.

    Assume now that $A\subset B$ and $B\subset C$ are Galois extension and that $A\subset C$ is an $\mc{F}$-extension.
    We consider the tower $A[C^p]\subset B\subset C$.
    Since $B\subset C$ is Galois,  $C$ is projective as $B$-module, and since $C$ is a projective $A[C^p]$-module,
     $B$ has to be a projective $A[C^p]$-module by Proposition \ref{tower of projective extensions}. 
     
     We now consider the tower $A\subset A[C^p]\subset B$. Since $A\subset B$ is Galois and $B$ is a projective $A[C^p]$-module, both extensions $A\subset A[C^p]$ and $A[C^p]\subset B$ are Galois by Proposition \ref{exponent one purely inseparable towers}.	Finally, as $A[C^p]\subset B$ and $B\subset C$ are Galois, and $A[C^p]\subset C$ has exponent one, we conclude that $A[C^p]\subset C$ is Galois, by Lemma \ref{lem: basics on Galois extensions}. 
     
     We have shown that both extensions $A\subset A[C^p]$ and $A[C^p]\subset C$ are Galois, so $A\subset C$ is purely inseparable.
\end{proof}

We now consider compositions of purely inseparable extensions of higher exponent. 
In Lemma \ref{lem: composition of Galois extensions}, we have shown that the composition of two Galois extensions is Galois if and only if it is an $\mc{F}$-extension.  The next two propositions address the problem of characterizing when the composition of a Galois extension and a purely inseparable extension induces again a purely inseparable extension. 

\begin{proposition}\label{prop: comp of pi extension and Galois extension}
    Let $A\subset B\subset C$ be a tower of ring extensions such that $A\subset B$ is purely inseparable and $B\subset C$ is Galois. 
    Then $A\subset C$ is purely inseparable if and only if $A\subset C$ is an $\mc{F}$-extension.
\end{proposition}
\begin{proof}
    If $A\subset C$ is purely inseparable, we already know that $A\subset C$ is an $\mc{F}$-extension. 

    Conversely, assume that $A\subset C$ is an $\mc{F}$-extension. We consider the tower $A\subset A[C^p]\subset B$. By hypothesis $A\subset B$ is purely inseparable, and $A\subset A[C^p]$ is an $\mc{F}$-extension. In addition, $C$ is a projective $A[C^p]$-module and a projective $B$-module by hypothesis, so $B$ is a projective $A[C^p]$-module by Proposition \ref{tower of projective extensions}. Hence we can apply Lemma \ref{lem: on purely inseparable towers with B<C of exponent one} and conclude that $A\subset A[C^p]$ is purely inseparable and $A[C^p]\subset B$ is Galois. Since by hypothesis $B\subset C$ is also Galois, we conclude by Lemma \ref{lem: basics on Galois extensions} that $A[C^p]\subset C$ is Galois. We have shown that $A\subset A[C^p]$ is purely inseparable and $A[C^p]\subset C$ is Galois, so $A\subset C$ is purely inseparable.
    \end{proof}

\begin{proposition}\label{prop: composition of a Galois and a purely inseparable}
    Let $A\subset B\subset C$ be a tower of ring extensions such that $A\subset B$ is Galois and $B\subset C$ is purely inseparable. 
    Then $A\subset C$ is purely inseparable if and only if $A\subset C$ is an $\mc{F}$-extension. Moreover, in this case, each of the extensions $A\subset B[C^{p^k}]$ is purely inseparable for all $k\geq 0$.
\end{proposition}
\begin{proof}
If $A\subset C$ is purely inseparable, then we know that $A\subset C$ is in particular an $\mc{F}$-extension. 
We now prove the converse. In fact, we prove by induction on $e:=\exp(C/B)$ that if $A\subset C$ is an $\mc{F}$-extension then $A\subset B[C^{p^k}]$ is purely inseparable for all $k$. 
If $e=1$, this is Lemma \ref{lem: composition of Galois extensions}. Assume that $e>1$ and consider the tower $A\subset B\subset B[C^p]$, where $A\subset B$ is Galois and $B\subset B[C^p]$ is purely inseparable of exponent $e-1$. We verify that $A\subset B[C^p]$ is an $\mc{F}$-extension. For $k\geq 1$ we have $A[(B[C^p])^{p^k}]=A[C^{p^{k+1}}]$ since $A\subset B$ has exponent one. By hypothesis, $C$ is projective over $A[C^{p^{k+1}}]$ and also over $B[C^p]$, so the latter is projective over  $A[C^{p^{k+1}}]$ by Proposition \ref{tower of projective extensions}. We can now apply the inductive hypothesis to the tower   $A\subset B\subset B[C^p]$ and conclude that $A\subset B[C^{p^k}]$ is purely inseparable for all $k\geq 1$. Finally, we consider the tower $A\subset B[C^p]\subset C$ where $A\subset B[C^p]$ is purely inseparable, $B[C^p]\subset C$ is Galois, and $A\subset C$ is an $\mc{F}$-extension. By Proposition \ref{prop: comp of pi extension and Galois extension}, $A\subset C$ is purely inseparable. We have now shown that $A\subset B[C^{p^k}]$ is purely inseparable for all $k\geq 0$.
This ends the proof of the proposition.
\end{proof}

Finally, as a corollary, we obtain Theorem \ref{composition of purely inseparable extensions} from the introduction,  which shows that the composition of two purely inseparable extensions induces a purely inseparable tower if we ask some intermediate subextensions to be $\mc{F}$-extensions.

\begin{corollary}\label{cor: composition of purely inseparable}
    Let $A\subset B$ and $B\subset C$ be purely inseparable extensions. If $A\subset B[C^{p^k}]$ is an $\mc{F}$-extension for all $k\geq 1$, then $A\subset C$ is purely inseparable.
\end{corollary}

\begin{proof}
    We prove this using an inductive argument. Let $e:=\exp(C/B)$. Since $B\subset C$ is purely inseparable, for all $0\leq k < e$, the intermediate subextensions $B[C^{p^{k+1}}]\subset B[C^{p^k}]$ are Galois. We analize the case in which $k=e-1$. Observe that $A\subset B$ is purely inseparable, $B\subset B[C^{p^{e-1}}]$ is Galois and $A\subset B[C^{p^{e-1}}]$ is an $\mc{F}$-extension by hypothesis, thus $A\subset B[C^{p^{e-1}}]$ is purely inseparable by Proposition \ref{prop: comp of pi extension and Galois extension}. Using the same argument and descending on $k$, we see that for each $0\leq k<e$ the extension $A\subset B[C^{p^{k+1}}]$ is purely inseparable, $B[C^{p^{k+1}}]\subset B[C^{p^{k}}]$ is Galois and $A\subset B[C^{p^k}]$ is an $\mc{F}$-extension, so $A\subset B[C^{p^k}]$ is purely inseparable. In particular, when we reach $k=0$, we conclude that $A\subset C$ is purely inseparable. 
\end{proof}

\section*{Acknowledgements}
We are deeply grateful to Orlando Villamayor for many helpful discussions and for his comments on the preparation of these notes. 
The second author would also like to thank the hospitality and excellent working conditions at the Department of Mathematics of the UAM, where he carried out part of this work as a Visiting Professor.

	The first author was partially supported by PID2022-138916NB-I00, by the Spanish Ministry of Science and Innovation, through the ``Severo Ochoa'' Programme for Centres of Excellence in R\&D (CEX2019-000904-S). The second author was partially supported by CONICET and MINCYT (PICT-2018-02073). Both authors were also partially supported by the Madrid Government under the multiannual Agreement with UAM in the line for the Excellence of the University Research Staff in the context of the V PRICIT (Regional Programme of Research and Technological Innovation) 2022-2024.

\bibliographystyle{abbrv} 
\bibliography{References}

\end{document}